\pdfoutput=1
\RequirePackage{ifpdf}
\ifpdf 
\documentclass[pdftex]{sigma}
\else
\documentclass{sigma}
\fi

\newcommand{\so}{\scriptscriptstyle \rm I}
\newcommand{\st}{\scriptscriptstyle \rm I\hspace{-1pt}I}
\newcommand{\sth}{\scriptscriptstyle \rm I\hspace{-1pt}I\hspace{-1pt}I}
\newcommand{\sso}{\scriptscriptstyle \rm i}
\newcommand{\sst}{\scriptscriptstyle \rm i\hspace{-1pt}i}

\newcommand{\bu}{\bar u}
\newcommand{\bv}{\bar v}

\newcommand{\bt}{{\bar t}}
\newcommand{\btt}{{\bar\tau}}

\newcommand{\bw}{\bar w}

\newcommand{\be}[1]{\begin{equation}\label{#1}}
\newcommand{\ba}[1]{\begin{multline}\label{#1}}
\newcommand{\ee}{\end{equation}}
\newcommand{\ea}{\end{multline}}

\newtheorem{thm}{Theorem}[section]
\newtheorem{prop}[thm]{Proposition}
\newtheorem{lem}[thm]{Lemma}
\newtheorem{cor}[thm]{Corollary}

\theoremstyle{definition}
\newtheorem{rem}[thm]{Remark}
\newtheorem{exam}[thm]{Example}
 \makeatletter
 \@addtoreset{equation}{section}
 \makeatother
 
\newcommand{\bea}{\begin{eqnarray}}
\newcommand{\eea}{\end{eqnarray}}

\def\BB{{\mathbb{B}}}
\def\CC{{\mathbb{C}}}
\def\hBB{\hat{\mathbb{B}}}

\def\ZZ{{\mathbb Z}}
\def\TT{{\mathbb{T}}}

\def\rvac{|0\rangle}

\def\Ee{{\sf e}}

\def\r#1{\eqref{#1}}
\def\sk#1{\left(#1\right)}

\def\fgo{\mathfrak{f}}

\def\ot{\otimes}

\def\sT{{\sf T}}

\def\Ig{I_{\mathfrak{g}}}

\def\rvec{|0\rangle}
\def\lvec{\langle0|}

\def\hc{{\sf h}}

\def\BBo{\mathbb{B}_{\mathfrak{o}}}
\def\BBg{\mathbb{B}_{\mathfrak{gl}}}
\def\cBo{\mathcal{B}_{\mathfrak{o}}}
\def\cBg{\mathcal{B}_{\mathfrak{gl}}}

\def\vk{\chi}
\def\Z{\mathcal{Z}}
\def\Oml{\Omega^L}
\def\Omr{\Omega^R}
\def\Gml{\Gamma^L}
\def\Gmr{\Gamma^R}

\def\eigen{\tau}
\def\lp{{\ell'}}
\def\kp{{k'}}
\def\ip{{i'}}
\def\jp{{j'}}

\def\hlam{\hat{\lambda}}
\def\halp{\hat{\alpha}}

\begin{document}

\allowdisplaybreaks

\newcommand{\arXivNumber}{2412.05224}

\renewcommand{\PaperNumber}{078}

\FirstPageHeading

\ShortArticleName{Rectangular Recurrence Relations in $\mathfrak{gl}_{n}$ and $\mathfrak{o}_{2n+1}$ Invariant Integrable Models}

\ArticleName{Rectangular Recurrence Relations in $\boldsymbol{\mathfrak{gl}_{n}}$\\ and $\boldsymbol{\mathfrak{o}_{2n+1}}$ Invariant Integrable Models}

\Author{Andrii LIASHYK~$^{\rm a}$, Stanislav PAKULIAK~$^{\rm b}$ and Eric RAGOUCY~$^{\rm b}$}

\AuthorNameForHeading{A.~Liashyk, S.~Pakuliak and E.~Ragoucy}

\Address{$^{\rm a)}$~Beijing Institute of Mathematical Sciences and Applications (BIMSA),\\
\hphantom{$^{\rm a)}$}~No.~544, Hefangkou Village Huaibei Town, Huairou District Beijing 101408, P.R.~China}
\EmailD{\href{mailto:liashyk@bimsa.cn}{liashyk@bimsa.cn}}

\Address{$^{\rm b)}$~Laboratoire d'Annecy-le-Vieux de Physique Th\'eorique (LAPTh),\\
\hphantom{$^{\rm b)}$}~Chemin de Bellevue, BP 110, F-74941, Annecy-le-Vieux Cedex, France}
\EmailD{\href{mailto:pakuliak@lapth.cnrs.fr}{pakuliak@lapth.cnrs.fr}, \href{mailto:ragoucy@lapth.cnrs.fr}{ragoucy@lapth.cnrs.fr}}

\ArticleDates{Received February 25, 2025, in final form September 01, 2025; Published online September 21, 2025}

\Abstract{A new method is introduced to derive general recurrence relations for off-shell Bethe vectors in quantum integrable models with either type $\mathfrak{gl}_n$ or type $\mathfrak{o}_{2n+1}$ symmetries. These recurrence relations describe how to add a single parameter $z$ to specific subsets of Bethe parameters, expressing the resulting Bethe vector as a linear combination of monodromy matrix entries that act on Bethe vectors which do not depend on $z$. We refer to these recurrence relations as \textit{rectangular} because the monodromy matrix entries involved are drawn from the upper-right rectangular part of the matrix. This construction is achieved within the framework of the zero mode method.}


\Classification{82B23; 81R12; 17B37; 17B80}

\section{Introduction}

We will consider generic $\mathfrak{g}$-invariant quantum integrable models
in the framework of the algebraic Bethe ansatz \cite{Fad96}. In such models the
monodromy matrix depends on a spectral parameter and satisfies the $RTT$ relations \cite{FRT} with a $\mathfrak{g}$-invariant
$R$-matrix. In the Hilbert space of physical states,
one can always construct from the local operators of the model a basis of states
 which are eigenvectors of a set of commuting Hamiltonians. These states form a representation of the finite-dimensional Lie
algebra $\mathfrak{g}$. In the algebraic
Bethe ansatz, they are constructed from the monodromy matrix entries
which depend on spectral parameters satisfying the Bethe equations.
These states are called on-shell Bethe vectors. When the spectral parameters
are generic (not forced to obey the Bethe equations), the Bethe vectors are called off-shell and their combinatorial properties
are defined solely by the $RTT$ relation with a given $R$-matrix.
As a~consequence, one can replace the monodromy matrix of a generic
model by the fundamental $T$-operator of the Yangian $Y(\mathfrak{g})$
in its matrix realization \cite{Drinfeld,Molev}.
The commutation relation for this $T$-operator~$T(u)$,
which depends on a formal spectral parameter $u$, coincides with
the commutation relations of monodromy matrix in a generic $\mathfrak{g}$-invariant
integrable model and the generators of the finite dimensional symmetry
may be identified with the zero modes of the matrix entries~$T_{i,j}(u)$
of the Yangian fundamental $T$-operator. Below we will explore the
Yangian fundamental $T$-operator~$T(u)$ calling it the monodromy
matrix of a generic $\mathfrak{g}$-invariant integrable model.\looseness=-1

One of the key problems in quantum integrable models is the presentation
of the Bethe vector scalar products in a determinant form.
For periodic boundary conditions, this problem was
investigated and fully solved in \cite{S-Det89}
for the integrable system associated with the simplest Lie algebra $\mathfrak{gl}_2$,
where the structure of the Bethe vectors is quite simple.

In the models associated to higher rank symmetries the structure of the Bethe vectors
is rather complicated. In principle, the nested Bethe ansatz (\cite{KR1, KR2} for $\mathfrak{gl}_n$ and \cite{Res91} for $\mathfrak{o}_n$) ensures that the
Bethe vectors can be expressed as a combinatorial expression
of the monodromy matrix entries acting on a vacuum vector.
However, whether there is a determinant form for scalar products of Bethe vectors in the general case remains an open question and has answers only in the cases $\mathfrak{gl}_3$ \cite{BPRS12} and \smash{$\mathfrak{gl}_{2|1}$} \cite{HLPRS16}.

It is worth noting that there is another method for studying eigenvectors and correlation functions in integrable systems, which is based on the so-called quantum separation of variables (SoV) method \cite{Sk90, Sk95}. Recently,
for models associated with the Lie algebra $\mathfrak{gl}_n$ was proposed~\cite{GLMS17} another construction to describe eigenvector using only one creation operator $B(u)$ closely related to SoV.
A little later, significant progress has been made in SoV method \cite{MN18}, where it was proposed to describe vectors in terms of the actions of transfer matrices on a certain vector, which is a fairly universal construction.
Both groups develop their methods significantly~\cite{GLMR20, MNV20, RV21} and managed to use this result to describe some overlaps of Bethe vectors.
However, the study of correlation functions in this approach is far from being complete and, up to now, applies mostly to $\mathfrak{gl}_n$ models.

On the other hand, if one can find formulas for the action of the
monodromy matrix entries on the off-shell Bethe vectors, as well as recurrence relations
for them, then one can find expressions for the scalar products of off-shell Bethe vectors, and ultimately obtain recurrence relations for the building blocks
of these scalar products. Expressions for the scalar product and the norm of Bethe vectors were obtained in \cite{IzeK84,Kor82} for $\mathfrak{gl}_2$ invariant models and in \cite{Res86} for $\mathfrak{gl}_3$ invariant models.
General expressions for the scalar product, the norm and the recurrence relations were achieved in \cite{HLPRS-SP17,HLPRS-SP18} for models associated to $\mathfrak{gl}(m|n)$.
Here by recurrence relations for the off-shell Bethe
vectors we mean the possibility to express a Bethe vector
with an extended set of Bethe parameters (for example, with an extension by a parameter $z$)
as a linear combination of the action of monodromy matrix entries $T_{i,j}(z)$ with
$i<j$ acting on off-shell Bethe vectors which do not depend on $z$.

The action formulas of the monodromy entries on the off-shell Bethe vectors
in $\mathfrak{gl}(m|n)$- and $\mathfrak{o}_{2n+1}$-invariant integrable models
were given in \cite{HLPRS17b} and \cite{LP2}, respectively.
The existence of recurrence relations is less investigated. Some examples
of such relations can be found in the aforementioned papers. One of
the goals of this paper is to fill this gap and produce all possible
recurrence relations for the Bethe vectors in the models
with $\mathfrak{gl}_{n}$ and $\mathfrak{o}_{2n+1}$ symmetries.

In the present paper, we show that in order to solve this problem, one has to combine a~single
(simple) action of the monodromy matrix entry with the action of the zero modes of the
monodromy matrix, the latter being identified with the simple root generators of the Lie algebra
$\mathfrak{g}$. We first perform the calculation of the recurrence relations
using this method for the Bethe vectors in $\mathfrak{gl}_n$-invariant models
and then extend this approach to the Bethe vectors in
$\mathfrak{o}_{2n+1}$-invariant models. The obtained recurrence relations will
be then tested in several limiting cases to observe the reductions over rank of
$\mathfrak{g}$ and embeddings of $Y(\mathfrak{gl}_{\ell})\otimes
Y(\mathfrak{gl}_{n-\ell})$ into $Y(\mathfrak{gl}_{n})$ and
 of $Y(\mathfrak{o}_{2\ell+1})\otimes Y(\mathfrak{gl}_{n-\ell})$
into $Y(\mathfrak{o}_{2n+1})$.

This paper is motivated by two key objectives. First, we would like to
generalize the nested Bethe ansatz to the cases when not only extreme nodes
of the Dynkin diagram are singularized, but also when the singularized node is inside the
Dynkin diagram. In the case of the $\mathfrak{gl}_n$ algebra, some progress in this direction was made in the recent papers \cite{KT1,KT2} in the framework of the trace formula presentation of off-shell Bethe vectors. To the best of our knowledge,
for other series, the nested Bethe ansatz method was developed
 only in the cases when the singularized node is an extreme node of $\mathfrak{gl}_n$-type.
 In our approach, the recurrence relations for the off-shell Bethe vectors in
 $\mathfrak{o}_{2n+1}$-invariant integrable models are generalization of the
 nested Bethe ansatz for arbitrary singularized node in the Dynkin diagram
 of Lie algebra $\mathfrak{o}_{2n+1}$. The second motivation is to develop
 methods to investigate the recurrence relations and analytical properties
 of the highest coefficients in the Bethe vector scalar product in
 $\mathfrak{g}$-invariant integrable models. An application of our approach
 to the properties of highest coefficients in $\mathfrak{o}_{2n+1}$-invariant
 integrable models can be found in the paper \cite{LPR25}.

Theorems~\ref{rec-pr} and \ref{main-th} are the main results of the paper.
It is surprising that these rectangular relations were not known previously,
although some examples of these recurrence relations for the Bethe
vectors in $\mathfrak{gl}_n$-invariant integrable models were obtained earlier
in \cite{G23,HLPRS-SP17}.

The plan of the paper is as follows. In Section \ref{sect:prelim}, we recall some basic algebraic notions
on Yangians, which are at the core of the models under consideration, as well as some definitions on Bethe vectors.
Section \ref{sect:gln} deals with the rectangular recurrence relations for $\mathfrak{gl}_n$ models. Here we verify also the consistency of the recurrence
relations with the embedding $Y(\mathfrak{gl}_{\ell})\otimes
Y(\mathfrak{gl}_{n-\ell})$ into~$Y(\mathfrak{gl}_{n})$ and compare our approach with
the recent papers \cite{KT1,KT2}.
The case of $\mathfrak{o}_{2n+1}$ is studied in Section \ref{sect:oN}. The main result is contained in Section \ref{sect:rec-oN}, where
rectangular recurrence relations for $\mathfrak{o}_{2n+1}$ models are presented.
Some particular cases and examples are displayed in Section \ref{sect:ex-rec}.
We also show that the recurrence relations are consistent with the embedding of $Y(\mathfrak{gl}_n)$
or $Y(\mathfrak{o}_{2a+1})\otimes Y(\mathfrak{gl}_{n-a})$ in
$Y(\mathfrak{o}_{2n+1})$: the Bethe vectors exhibit a nice factorisation property over these subalgebras (see Sections
\ref{sect:red-gln} and \ref{sect:red-o.x.gl}). We conclude in Section \ref{conclu}, and two appendices are devoted to the technical proofs.

\section{Preliminaries\label{sect:prelim}}

\subsection{Algebraic context}
Let $\mathfrak{g}$ be either the classical Lie algebra $\mathfrak{gl}_n$ or the
orthogonal algebra $\mathfrak{o}_{2n+1}$,
where $n=2,3,\dots$.
We will use the set of positive integers
$I_{\mathfrak{gl}_n}=\{1,\dots,n\}$ to index elements of the operators in~$\operatorname{End}({\CC}^{n})$ and the set of integers
$I_{\mathfrak{o}_{2n+1}}=\{-n,-n+1,\dots,-1,0,1,2,\dots,n\}$
to index elements of the operators in $\operatorname{ End}\bigl({\CC}^{2n+1}\bigr)$.
 We will use the notation $\Ig$ to describe the two sets of indices simultaneously.
Let $N=n$ and $N=2n+1$ for the algebras $\mathfrak{gl}_n$ and $\mathfrak{o}_{2n+1}$
respectively.

{\bf $\boldsymbol{RTT}$ presentation of the Yangians $\boldsymbol{Y(\mathfrak{gl}_n)}$ and $\boldsymbol{Y(\mathfrak{o}_{2n+1})}$.}
Let $R_{\mathfrak{g}}(u,v)$ be the $\mathfrak{g}$-invariant $R$-matrix \cite{Yang,ZZ79}
\begin{equation}\label{R-mat}
 R_{\mathfrak{g}}(u,v) = \mathbf{I}\otimes\mathbf{I} + \frac{c \mathbf{P}}{u-v} - \frac{c \mathbf{Q}}{u-v+c\kappa_n} ,
\end{equation}
where $\mathbf{I}=\sum_{i\in\Ig}\Ee_{i,i}$ is the identity operator acting in the space $\CC^{N}$ and $\Ee_{i,j}\in\operatorname{ End}\bigl(\CC^N\bigr)$ are~${N\times N}$ matrices with the only nonzero entry equal to 1 at
the intersection of the $i$-th row and $j$-th column. The operators $\mathbf{P}$
and $\mathbf{Q}$ act in
$\CC^{N}\ot\CC^{N}$. They read
\begin{equation*}\label{PQ}
\mathbf{P}=\sum_{i,j\in \Ig}\Ee_{i,j}\otimes\Ee_{j,i}, \qquad
\mathbf{Q}=\sum_{i,j\in \Ig}\Ee_{-i,-j}\otimes\Ee_{i,j} ,
\end{equation*}
and\footnote{The value of $\kappa_n=\infty$ for the
algebra $\mathfrak{gl}_n$ simply means that in this case the
$\mathfrak{g}$-invariant $R$-matrix does not contain the term
proportional to operator $\mathbf{Q}$ and so coincides with the Yang $R$-matrix
\cite{Yang}. The parameter $\kappa_n=n-1/2$ is relevant only for the algebra
$\mathfrak{g}=\mathfrak{o}_{2n+1}$.\label{fn:kappa}}
\[
\kappa_n=
\begin{cases}
\infty &\text{for } \mathfrak{g}=\mathfrak{gl}_{n} ,\\
n-1/2&\text{for } \mathfrak{g}=\mathfrak{o}_{2n+1} .
\end{cases}
\]

In the $\mathfrak{o}_{2n+1}$-case, for any matrix $M\in\operatorname{ End}\bigl(\CC^{2n+1}\bigr)$, we denote by $M^{\rm t}$ the transposition with respect to the secondary diagonal
\smash{$
\bigl(M^{\rm t}\bigr)_{i,j}=M_{-j,-i} $}.
In particular, $\mathbf{Q} = \mathbf{P}^{{\rm t}_{1}} = \mathbf{P}^{{\rm t}_{2}}$.

The generic $\mathfrak{g}$-invariant integrable model is described by the monodromy
matrix $T(u)$ which depends on the spectral parameter $u$ and satisfies the commutation
relation
\begin{equation}\label{RTT}
 R_{\mathfrak{g}}(u,v) \left( T(u)\otimes\mathbf{I} \right)
 \left( \mathbf{I}\otimes T(v) \right) =
 \left( \mathbf{I}\otimes T(v) \right) \left( T(u)\otimes\mathbf{I} \right)
 R_{\mathfrak{g}}(u,v) ,
\end{equation}
with the $\mathfrak{g}$-invariant $R$-matrix \r{R-mat}.

Equation \eqref{RTT} yields the commutation relations of the monodromy matrix entries
\begin{align}
 [ T_{i,j}(u), T_{k,l}(v) ] ={}&
 \frac{c}{u-v}( T_{k,j}(v)T_{i,l}(u) - T_{k,j}(u) T_{i,l}(v) )\nonumber\\
 &+\frac{c}{u-v+c\kappa_n}\sum_{p=-n}^n(\delta_{k,-i}
 T_{p,j}(u)T_{-p,l}(v)-
 \delta_{l,-j} T_{k,-p}(v)T_{i,p}(u)).\label{rtt}
\end{align}
The parameter $c$ in \r{R-mat} and \r{rtt} is a Yangian deformation parameter.
In the limit $c\to 0$, the Yangian goes to the Borel subalgebra of the
corresponding
loop algebra. When non-zero, this parameter can be always changed to the value $c=1$
by rescaling formal spectral parameters $u$ and $v$ but for our convenience we prefer to keep it arbitrary.
Note that the second line in \eqref{rtt} occurs only in the case of $\mathfrak{o}_{2n+1}$ models, in accordance with the footnote \ref{fn:kappa}.
We will assume the following dependence of the monodromy matrix on the
spectral parameter
\begin{equation}\label{genser}
T_{i,j}(u)=\vk_i \delta_{ij}+\sum_{m\geq0} T_{i,j}[m](u/c)^{-m-1} ,
\end{equation}
where the parameters $\vk_i$, $i\in\Ig$, are twisting parameters.
For $\mathfrak{g}=\mathfrak{gl}_N$, these parameters are all independent, while for $\mathfrak{g}=\mathfrak{o}_{2n+1}$ they satisfy the relation
$\vk_i\vk_{-i}=1$ and $\vk_0 = 1$.
For both algebras, the possibility to start
the series expansion \r{genser}
with $\chi_i \delta_{ij}$ instead of $\delta_{ij}$
corresponds to the usual twist $T(z)\to D T(z)$ where $D$ is the diagonal matrix with $D_{ii}=\chi_i$.
These twisting maps are isomorphisms between two RTT-algebras, since the $R$-matrix commutes with $D\otimes D$ (with the restriction $D^{\rm t} = D$ in the $\mathfrak{o}_{2n+1}$ case).

{\bf Central element for $\boldsymbol{\mathfrak{g}=\mathfrak{o}_{2n+1}}$.}
Remark that for $\mathfrak{g}=\mathfrak{o}_{2n+1}$,
the pole at $u=v-c\kappa_n$ in the $R$-matrix \r{R-mat}
and in the commutation relation \r{RTT} implies that the monodromy matrix should satisfy following relations \cite{JLM18}
$
T(z)^{\rm t}\cdot T(z+c\kappa_n)= T(z+c\kappa_n)\cdot T(z)^{\rm t}=
\mathcal{C}(z)= \mathbf{I}$,
where $\mathcal{C}(z)$ is a central operator in the algebra defined by the relations \eqref{RTT}.
Let $\pi(z)$ be a~formal series $\pi(z)=1+\sum_{m\geq0}\pi_m (c/z)^{-m-1}$ which
solves the equation
$
\mathcal{C}(z)=\pi(z) \pi(z+c\kappa_n)$.
This equation can be solved inductively expressing the
coefficient of the formal series $\pi(z)$ through the coefficients of the
central element $\mathcal{C}(z)$. Then one can rescale the monodromy matrix~${
T(z)\to T'(z) = \pi(z)^{-1} T(z)}
$
in such a way that the rescaled monodromy matrix satisfies the equation
\begin{equation}\label{cen3}
 T'(z)^{\rm t}\cdot T'(z+c\kappa_n)= T'(z+c\kappa_n)\cdot T'(z)^{\rm t} = \mathbf{I} .
\end{equation}

{\bf Zero modes.}
The zero modes $\sT_{i,j} := T_{i,j}[0]$ (see \eqref{genser}) will play an important role in
our approach.
Considering expansion of $T_{i,j}(u)$ and the rational functions in \r{rtt} as series with respect to $1/u$, the coefficient of $u^{-1}$ in \r{rtt} yields
\begin{equation}
\label{Azm}
[\sT_{i,j},T_{k,l}(v)]=\vk_i \delta_{i,l} \sT_{k,j}(v)- \vk_j \delta_{k,j} T_{i,l}(v)
\end{equation}
for $\mathfrak{g}=\mathfrak{gl}_n$ and
\begin{equation}
\label{zm1}
[\sT_{i,j},T_{k,l}(v)]=\vk_i (\delta_{i,l} T_{k,j}(v)-\delta_{l,-j} T_{k,-i}(v))
-\vk_j (\delta_{k,j} T_{i,l}(v)-\delta_{k,-i} T_{-j,l}(v))
\end{equation}
for $\mathfrak{g}=\mathfrak{o}_{2n+1}$.

More generally,
in each integrable models,
 the monodromy matrix zero mode operators may be always defined as the
 operators built from the monodromy matrix entries
 which
satisfy the commutation relations of the finite-dimensional algebra $\mathfrak{g}$.

{\bf Embeddings.}
We describe different embeddings in the Yangians $Y(\mathfrak{gl}_{n})$ and $Y(\mathfrak{o}_{2n+1})$ that will be reflected in a factorisation (also called splitting) property of the Bethe vectors, see Sections~\ref{sect:red-gl.x.gl} and~\ref{sect:red-o.x.gl} below.

{\bf Embedding $\boldsymbol{Y(\mathfrak{gl}_a)\otimes Y(\mathfrak{gl}_{n-a})
\hookrightarrow Y(\mathfrak{gl}_{n})}$.}
From the commutation relations \eqref{rtt}, it is clear that in $Y(\mathfrak{gl}_{n})$ the elements
$T_{i,j}(z)$, $1\leq i,j\leq a$ generate a Yangian subalgebra~$Y(\mathfrak{gl}_{a})$ in $Y(\mathfrak{gl}_{n})$, while the elements $T_{a+i,a+j}(z)$, $1\leq i,j\leq n-a$ generate the Yangian subalgebra~$Y(\mathfrak{gl}_{n-a})$. However, these two subalgebras do not commute. To get an embedding of~${Y(\mathfrak{gl}_a)\otimes Y(\mathfrak{gl}_{n-a})}$ in $Y(\mathfrak{gl}_{n})$, one needs to consider quantum minors,
see relation (1.84), in \cite[Corollaries 1.7.2 and 1.11.4]{Molev}.

{\bf Embeddings of $\boldsymbol{Y\!(\mathfrak{gl}_{n})\!\hookrightarrow \!Y(\mathfrak{o}_{2n+1})}$.}
There are several ways to embed $Y(\mathfrak{gl}_{n})$ in $Y(\mathfrak{o}_{2n+1})$. We describe here two of them that have some relevance for the study of Bethe vectors. Let us denote the $Y(\mathfrak{gl}_n)$
monodromy matrix embedded in $Y(\mathfrak{o}_{2n+1})$ as $\TT(z)$.
\begin{itemize}\itemsep=0pt
\item From the commutation relations \eqref{rtt}, it is clear that in $Y(\mathfrak{o}_{2n+1})$ the elements
\begin{equation}\label{Tgl}
\TT_{i,j}(z) := T_{i,j}(z) ,\qquad 1\leq i,j\leq n
\end{equation}
 generate the Yangian $Y(\mathfrak{gl}_{n})$, since they obey the
commutation relations \eqref{rtt} without the second line (all the terms
 which include the Kronecker symbol $\delta$ vanish).

\item
Another embedding can be done, considering the elements $T_{i,j}(z)$, $-n\leq i,j\leq -1$ and defining
\begin{equation}\label{em-sh}
 \hat{\TT}_{i,j}(z) := T_{i - n - 1,j - n -1}(z), \qquad 1\leq i,j\leq n.
\end{equation}
It is easy to see that $\hat{\TT}_{i,j}(z)$ also obey the commutation relations \eqref{rtt} without the second line.
The twisting parameters for the monodromy $\hat{\TT}(z)$ in this case are $(\chi_{n+1-i})^{-1}$.
\end{itemize}

{\bf Embedding $\boldsymbol{Y(\mathfrak{o}_{2a+1})\otimes Y(\mathfrak{gl}_{n-a})
\hookrightarrow Y(\mathfrak{o}_{2n+1})}$.}
This type of embedding is more intricate. It uses the concept of
quasi-determinants \cite{GelfandRetakh} \smash{$\widehat T_{i,j}(z)$} of the monodromy matrix, and we refer to, e.g., \cite{Molev} for more details about quasi-determinants in the Yangian case.
Indeed, one can show that $\widehat T_{i,j}(z)$, $1\leq i,j\leq a$ generate a Yangian subalgebra $Y(\mathfrak{o}_{2a+1})$,
while~${T_{k+a,l+a}(z)}$, $1\leq k,l\leq n-a$ generate a Yangian subalgebra $Y(\mathfrak{gl}_{n-a})$. Moreover, we have
\smash{$\big[\widehat T_{i,j}(z) , T_{k+a,l+a}(z)\big]=0$} for $1\leq i,j\leq a$ and $1\leq k,l\leq n-a$.
A detailed presentation of this approach is beyond the scope of the present article, we
refer to \cite[Theorem 3.7 and Corollary 3.10]{JLM18} for a detailed presentation of the construction.
It is remarkable that this type of embedding
is reflected in a simple way
on the structure of Bethe vectors, see Section \ref{sect:red-o.x.gl} below.

\subsection{Bethe vectors}

In the framework of algebraic Bethe ansatz, the states in the
Hilbert space of the physical model are defined by the vectors $\rvec$ and $\lvec$
such that
\begin{equation}\label{rvec}
T_{i,j}(z)\rvec =0 ,\quad i>j ,\qquad T_{i,i}(z)\rvec=\lambda_i(z)\rvec ,\quad i\in\Ig
\end{equation}
and
\begin{equation}\label{lvec}
\lvec T_{i,j}(z) =0 ,\quad i<j ,\qquad \lvec T_{i,i}(z)=\lambda_i(z)\lvec ,\quad i\in\Ig .
\end{equation}
In \r{rvec}, the monodromy matrix elements are acting to the right, while in \r{lvec}
they are acting to the left. If they exist, such vectors are called {\it vacuum vectors}.

The functions $\lambda_i(z)$ are characterizing the physical model under
consideration. Since we are
considering generic model, we will consider these functions as free functional
parameters. For $\mathfrak{gl}_n$-invariant
integrable models, the functions $\lambda_i(z)$ are all independent, while for
$\mathfrak{o}_{2n+1}$-invariant models and due to \r{cen3}
they satisfy the relations \cite{LP1}
\begin{equation}\label{lam}
\lambda_{-j}(z)=\frac{1}{\lambda_{j}(z_j)}\prod_{s=j+1}^n\frac{\lambda_s(z_{s-1})}{\lambda_s(z_s)},\qquad j=0,1,\dots,n ,
\end{equation}
where we introduced the \textit{shifted spectral parameter}
\begin{equation}\label{zs}
z_s=z-c\left(s-\frac12\right),\qquad s=0,1,\dots,n .
\end{equation}

Due to our choice \r{genser} of dependence of monodromy matrix entries
on the formal spectral parameter, the free functional parameters $\lambda_i(z)$
are formal series
\[
\lambda_i(z)=\vk_i+\sum_{\ell\geq 0}\lambda_i[\ell] (z/c)^{-\ell-1}
\]
with respect to the formal parameter $z$.

The Bethe vectors
 in the integrable model
depends on a collection of sets
\[
\bt=\begin{cases}
\bigl\{\bt^1,\dots,\bt^{n-1}\bigr\}&\text{for}\  \mathfrak{g}=\mathfrak{gl}_n,\\
\bigl\{\bt^0,\bt^1,\dots,\bt^{n-1}\bigr\}&\text{for}\  \mathfrak{g}=\mathfrak{o}_{2n+1}.
\end{cases}
\]
Here, the set $\bt^s=\{t^s_1,\dots,t^s_{r_s}\}$ denotes a collection of $r_s$
Bethe parameters. The non-negative integer~$r_s$ is the cardinality $|\bt^s|$ of the set
$\bt^s$.
The superscripts on the (sets of) Bethe parameters denotes their color.
 The colors are in correspondence with the simple roots of the algebra
 $\mathfrak{g}$.
This can be formalized through the operators $\mathsf{h}_i$ defined by
\begin{equation}\label{Cartan}
\vk_i \hc_i=\sT_{i,i}-\lambda_i[0]
\end{equation}
such that the vacuum vectors have zero eigenvalue
$
\hc_i \rvec =0$, $ \lvec \hc_i=0 $.
According to \r{Azm} and \r{zm1}, the monodromy matrix entries
$T_{k,l}(z)$ are eigenvector for the adjoint action of the operators
$\hc_i$
$
[\hc_i,T_{k,l}(z)]= (\delta_{i,l}-\delta_{i,k}) T_{k,l}(z)
$
for $\mathfrak{gl}_N$-invariant monodromies
and
\[
[\hc_i,T_{k,l}(z)]=(\delta_{i,l}-\delta_{i,k}+\delta_{i,-k}-\delta_{i,-l}) T_{k,l}(z)
\]
for $\mathfrak{o}_{2n+1}$-invariant monodromies. Using the operators $\hc_i$ \r{Cartan}, one also defines the operators
\begin{equation}\label{col-op}
\mathsf{t}_s=\sum_{i=s+1}^n \hc_i ,
\end{equation}
where $s=1,\dots,n-1$ for $\mathfrak{gl}_N$-invariant monodromies
and $s=0,1,\dots,n-1$ for $\mathfrak{o}_{2n+1}$-invariant monodromies.

 The Bethe vectors themselves are certain
 polynomials of the non-commutative monodromy entries $T_{i,j}(u)$ for $i <j$ depending on various Bethe parameters acting on the right vacuum vector $\rvec$
\begin{equation}\label{bvgr}
\BB\bigl(\bar t\bigr)=\mathcal{P}\bigl(T_{i< j}
\bigl(\bar t\bigr)\bigr)
 \rvec=\mathcal{B}\bigl(\bar t\bigr) \rvec,
\end{equation}
where the polynomial $\mathcal{B}\bigl(\bar t\bigr)=\mathcal{P}\bigl(T_{i< j}\bigl(\bar t\bigr)\bigr)$ is called a \textit{pre-Bethe vector}.

Analogously, left or dual Bethe vectors
are polynomials of monodromy entries $T_{i,j}(u)$ for~${i > j}$ acting to the left vacuum vector $\rvec$
\begin{equation}\label{bvgl}
\CC\bigl(\bar t\bigr)=\lvec \mathcal{P}'\bigl(T_{i > j}\bigl(\bar t\bigr)\bigr)=\lvec \mathcal{C}\bigl(\bar t\bigr).
\end{equation}
The polynomials $\mathcal{P}$ and $\mathcal{P}'$ are related by the transposition antihomomorphism (see Remark~\ref{rem21}).
The ordering of the non-commutative entries $T_{i,j}(t^s_a)$ in the polynomials
$\mathcal{P}$ and $\mathcal{P}'$ and the structure
 of these polynomials can be fixed in the framework of the
nested Bethe ansatz \cite{KR1, KR2} or by the method of projections
\cite{HLPRS17,KhP-Kyoto}. When the Bethe parameters
are generic, we call such Bethe vectors \r{bvgr} and \r{bvgl} {\it off-shell Bethe vectors}.

 The off-shell Bethe vectors $\BB\bigl(\bt\bigr)$ and
$\CC\bigl(\bt\bigr)$ are eigenvectors of the operators
$\mathsf{t}_s$ \r{col-op}
\[
\mathsf{t}_s\cdot\BB\bigl(\bt\bigr)=r_s \BB\bigl(\bt\bigr),\qquad \CC\bigl(\bt\bigr)\cdot\mathsf{t}_s=
r_s \CC\bigl(\bt\bigr),\qquad i=1,\dots,n,
\]
with $s=1,\dots,n-1$ for $\mathfrak{gl}_n$
and $s=0,1,\dots,n-1$ for $\mathfrak{o}_{2n+1}$. This property can be proved
using the recurrence relations for the Bethe vectors and the
action of the monodromy matrix entries $T_{i,i}(z)$ on Bethe vectors
(see proof of Proposition~3.1 in the case of $\mathfrak{o}_{2n+1}$ in \cite{LPR25}).

\begin{rem}\label{rem21}
In what follows, we will consider only the Bethe vectors
$\BB\bigl(\bt\bigr)$. All the relations for the dual Bethe vectors $\CC\bigl(\bt\bigr)$
can be obtained from the corresponding relations for $\BB\bigl(\bt\bigr)$
using the transposition antihomomorphism.
When we consider different embeddings, we will
use the notation $\BBg\bigl(\bt\bigr)$ and $\BBo\bigl(\bt\bigr)$ to distinguish the
off-shell Bethe vectors in the models with different symmetries.
But most often we will the use notation $\BB\bigl(\bt\bigr)$
since it will be clear from the context what type of Bethe
vector we are exploring.
\end{rem}

\begin{rem}\label{rem: chi homogeneity}
The commutation relations \eqref{rtt} between the monodromy matrix entries do not depend on the parameters $\chi_{i}$ explicitly, nor does the definition of the vacuum state \eqref{bvgr}. Since the Bethe vectors are polynomials in the entries of the monodromy matrix, they do not depend explicitly on the parameters $\chi_{i}$ either.
It is only when using the expansion \eqref{genser} that the dependence in the $\chi_i$ parameters becomes explicit, as for instance in the relations \eqref{Azm} or \eqref{zm1} which involve the zero mode action. Therefore, any relation involving only the Bethe vectors, the entries $T_{ij}(z)$ and/or the eigenvalues $\lambda_i(z)$ should not depend on the $\chi$'s. In other words, if the $\chi_i$ parameters appear explicitly in such relation, each coefficient of the parameters~$\chi_i$ should be set to zero independently. We will use this property in Appendix~\ref{ApB}.
\end{rem}

{\bf On-shell Bethe vectors.}
Due to the $RTT$ commutation relations \r{RTT}, the trace
of the monodromy matrix (the transfer matrix)
\begin{equation}\label{trace}
\mathcal{T}(z)=\sum_{i\in\Ig}T_{i,i}(z)
\end{equation}
commutes
$
\mathcal{T}(z) \mathcal{T}(z')=\mathcal{T}(z') \mathcal{T}(z)
$
for two different formal spectral parameters $z$ and $z'$. Upon expansion in $z$, the transfer matrix~\eqref{trace} generates a family of commuting operators.
The Bethe vectors become eigenvectors of the transfer matrix (also known as {\it on-shell Bethe vectors}) if the Bethe
parameters satisfy so called Bethe equations (see below \r{BE}).
To describe the Bethe equations in the $\mathfrak{g}$-invariant integrable models,
we introduce the rational functions
\[
f(u,v)=1+g(u,v)=h(u,v)g(u,v)=\frac{u-v+c}{u-v},\qquad \fgo(u,v)=\frac{u-v+c/2}{u-v} .
\]
 We also define the functions
\[
 \alpha_s(z) = \frac{\lambda_s(z)}{\lambda_{s+1}(z)}, \qquad s = 0, \dots, n-1.
\]

We will use the following convention for the products of scalar
functions depending on sets of parameters, for example,
\begin{equation}\label{eq:convention}
\lambda_s(\bt^s)=\prod_{t^s_a \in \bt^s}\lambda_s(t^s_a),\qquad
f\bigl(\bt^s,\bt^{s'}\bigr)=\prod_{t^s_a \in \bt^s}
\prod_{t^{s'}_b \in \bt^{s'}}f\bigl(t^s_a, t^{s'}_b\bigr),\qquad \mbox{etc.}
\end{equation}
with $f\bigl(\varnothing,\bt\bigr)=f\bigl(\bt,\varnothing\bigr)=1$.

The recurrence relations for the off-shell Bethe vectors will
be written as sums over partitions of the sets of Bethe parameters.
A {\it partition} $ \bigl\{\bt^s_{\so},\bt^s_{\st},\bt^s_{\sth}\bigr\}\vdash \bt^s$
corresponds to a decomposition into (possibly empty) disjoint subsets
$\bt^s_{\so}$, $\bt^s_{\st}$, $\bt^s_{\sth}$ such that
$\bt^s=\bt^s_{\so}\cup\bt^s_{\st}\cup\bt^s_{\sth}$
and \smash{$\bt^s_{\so}\cap\bt^s_{\st}=\bt^s_{\so}\cap\bt^s_{\sth}=\bt^s_{\st}\cap\bt^s_{\sth}=\varnothing$}.
The cardinalities of the subsets satisfy the equality
$
\big|\bt^s\big|=\big|\bt^s_{\so}\big|+\big|\bt^s_{\st}\big|+\big|\bt^s_{\sth}\big|$,
where some of the cardinalities $\big|\bt^s_{\so}\big|$, $\big|\bt^s_{\st}\big|$,
$\big|\bt^s_{\sth}\big|$ can be zero.
The partition $\bigl\{\bt^s_{\so},\bt^s_{\st}\bigr\}\vdash\bt^s$ is
defined analogously.

The on-shell Bethe vectors are eigenvectors of the transfer matrix
\begin{equation}\label{eq:transf}
\mathcal{T}(z)\cdot\BB\bigl(\bt\bigr)=\eigen\bigl(z;\bt\bigr) \BB\bigl(\bt\bigr)
\end{equation}
if each set $\bt^s$ of Bethe parameters satisfy {\it the Bethe equations}
\cite{KR2,Res91}
\begin{gather}
\alpha_s\bigl(\bt^s_{\so}\bigr)=\frac{\lambda_s\bigl(\bt^s_{\so}\bigr)}{\lambda_{s+1}\bigl(\bt^s_{\so}\bigr)}=
\frac{f_s\bigl(\bt^s_{\so},\bt^s_{\st}\bigr)}{f_s\bigl(\bt^s_{\st},\bt^s_{\so}\bigr)}
\frac{f\bigl(\bt^{s+1},\bt^s_{\so}\bigr)}{f\bigl(\bt^s_{\so},\bt^{s-1}\bigr)}
\qquad\!
\text{with } \begin{cases}s=1,\dots,n-1 &\text{for } \mathfrak{gl}_{n},\\
s=0,1,\dots,n-1 &\text{for } \mathfrak{o}_{2n+1},\end{cases}\!\!\!\label{BE}
\end{gather}
for any disjoint partition $\bigl\{\bt^s_{\so},\bt^s_{\st}\bigr\}\vdash\bt^s$ and with the conditions $\bt^{0} = \bt^n = \varnothing$ for $\mathfrak{gl}_{n}$ and $\bt^{-1} = \bt^n = \varnothing$ for $\mathfrak{o}_{2n+1}$.
In \r{BE}, we have used the convention \eqref{eq:convention}
for products of functions
and
the functions $f_s(u,v)$ are defined as
\[
f_s(u,v)=\begin{cases}\fgo(u,v),  &s=0,\\f(u,v),  &s=1,\dots,n-1 .
\end{cases}
\]
For the
$\mathfrak{gl}_{n}$-invariant models, the eigenvalue $\eigen\bigl(z;\bt\bigr)$ in \eqref{eq:transf}
is equal to \cite{KR2}
\[
\eigen\bigl(z;\bt\bigr)=\sum_{s=1}^n \lambda_s(z) f\bigl(\bt^s,z\bigr) f\bigl(z,\bt^{s-1}\bigr) .
\]
In the case of $\mathfrak{o}_{2n+1}$-invariant models, this eigenvalue is
\cite{LP2,Res91}
\begin{align*}
\eigen\bigl(z;\bt\bigr)={}&\lambda_0(z) f\bigl(\bt^0,z_0\bigr) f\bigl(z,\bt^0\bigr)+
\sum_{s=1}^n \bigl(\lambda_s(z) f\bigl(\bt^s,z\bigr) f\bigl(z,\bt^{s-1}\bigr)\\
&
+\lambda_{-s}(z) f\bigl(\bt^{s-1},z_{s-1}\bigr) f\bigl(z_s,\bt^{s}\bigr)\bigr) ,
\end{align*}
where we have used the notations \r{zs} and
 $\lambda_{-s}(z)$ satisfy the relations \r{lam}.

\section[Bethe vectors in gl\_n-invariant models]{Bethe vectors in $\boldsymbol{\mathfrak{gl}_n}$-invariant models}\label{sect:gln}

Denote by $\bigl\{\bt^s\bigr\}_i^{j}:=\bigl\{\bt^i,\bt^{i+1},\dots,\bt^{j}\bigr\}$
 a partial collection of sets $\bt^s$ of Bethe parameters
 for~$i\leq s\leq j$.
 For example, \smash{$\bt=\bigl\{\bt^s\bigr\}_1^{n-1}$} for $\mathfrak{gl}_n$ Bethe vectors.
We will always assume that the partial collection of sets
\smash{$\bigl\{\bt^s\bigr\}_i^{j}$} is empty if $j<i$.

For $1\leq \ell<n-1$ and $1<k\leq n$, we introduce the functions
\[
\psi_\ell\bigl(z;\bt\bigr)=g\bigl(z,\bt^{\ell-1}\bigr) h\bigl(\bt^\ell,z\bigr),\qquad \phi_k\bigl(z;\bt\bigr)=
h\bigl(z,\bt^{k-1}\bigr) g\bigl(\bt^k,z\bigr) .
\]
Remark that $\psi_1\bigl(z;\bt\bigr)=h\bigl(\bt^1,z\bigr)$ and $\phi_n\bigl(z;\bt\bigr)=h\bigl(z,\bt^{n-1}\bigr)$
because $g\bigl(z,\bt^0\bigr)=g\bigl(\bt^n,z\bigr)= 1$ since $\bt^0=\bt^n = \varnothing$.

\subsection[Action formula for T\_1,n(z) and T\_l+1,l]{Action formula for $\boldsymbol{T_{1,n}(z)}$ and $\boldsymbol{\sT_{\ell+1,\ell}}$}

Denote by $\bu^s=\bigl\{t^s_1,\dots,t^s_{r_s},z\bigr\}$ an extended set of the Bethe parameters of the color $s$. For shortness, we will write $\bu^s=\bigl\{\bt^s,z\bigr\}$.
The off-shell Bethe vectors
$\BB\bigl(\bt\bigr)$ in $\mathfrak{gl}_n$-invariant models are normalized in such a way that
the action of the monodromy entry $T_{1,n}(z)$ onto off-shell Bethe vector
$\BB\bigl(\bt\bigr)$ has the simple form\footnote{Note that normalization of the
off-shell Bethe vectors used in \r{Az1} differs from the normalization used
in the paper \cite{HLPRS17b}, but it is more suitable for the propose of this paper.}
\begin{equation}\label{Az1}
T_{1,n}(z)\cdot \BB\bigl(\bt\bigr)=\mu_1^n\bigl(z;\bt\bigr) \BB(\bu) ,
\end{equation}
where $\bu=\bigl\{\bu^1,\bu^2,\dots,\bu^{n-1}\bigr\}$.
The normalization factor $\mu_\ell^k\bigl(z;\bt\bigr)$ is defined as follows
for any $1\leq \ell<k\leq n$
\begin{equation}\label{Az2}
\mu_\ell^k\bigl(z;\bt\bigr)=\lambda_k(z) \psi_\ell\bigl(z;\bt\bigr) \phi_k\bigl(z;\bt\bigr)=\lambda_k(z)
g\bigl(z,\bt^{\ell-1}\bigr) h\bigl(\bt^\ell,z\bigr) h\bigl(z,\bt^{k-1}\bigr) g\bigl(\bt^k,z\bigr) .
\end{equation}

We denote by the symbol $\Z_{\ell}^k$ the operation of adding a parameter $z$ to the sets $\bt^\ell,\dots,\bt^{k-1}$
of Bethe parameters
in the off-shell Bethe vectors $\BB\bigl(\bt\bigr)$
\begin{equation}\label{Az3}
\Z_{\ell}^k\cdot \BB\bigl(\bt\bigr)=\BB\bigl(\bigl\{\bt^s\bigr\}_1^{\ell-1},\bigl\{\bt^s,z\bigr\}_\ell^{k-1},\bigl\{\bt^s\bigr\}_k^{n-1}\bigr) .
\end{equation}
For example, using this notation the action \r{Az1} can be written as follows:
\begin{equation}\label{Az4}
\Z_1^n\cdot\BB\bigl(\bt\bigr)=\frac{1}{\mu_1^n\bigl(z;\bt\bigr)} T_{1,n}(z)\cdot \BB\bigl(\bt\bigr) .
\end{equation}
We will use relation \r{Az1} in the form \r{Az4} as a base relation for
the inductive proof of Theorem~\ref{rec-pr}.

To find the rectangular recurrence relations for the off-shell Bethe vectors
$\BB\bigl(\bt\bigr)$, we will use the commutation relations \r{Azm} in the particular case
\begin{equation}\label{Az5}
[\sT_{\ell+1,\ell},T_{i,j}(z)]=\vk_{\ell+1} \delta_{\ell+1,j} T_{i,\ell}(z)-
\vk_{\ell} \delta_{\ell,i} T_{\ell+1,j}(z)
\end{equation}
and the action of the zero mode operators $\sT_{\ell+1,\ell}$ onto off-shell Bethe vectors.
These actions as well as the action \r{Az1} can be calculated in the framework of the
projection method. For more details, one can look at the \cite[Section~4.1]{LP2},
where the action of the zero modes on the off-shell Bethe vectors
in $\mathfrak{o}_{2n+1}$-invariant integrable model was calculated.
The calculation of the zero modes action for
the $\mathfrak{gl}_{n}$-invariant integrable model is similar, starting from the results stated in \cite{HLPRS17}.
We will write the action of the zero modes in the form
\begin{align}
\sT_{\ell+1,\ell}\cdot\BB\bigl(\bt\bigr) ={}&\sum_{\rm part}\bigl(\vk_{\ell+1}
\alpha_\ell\bigl(\bt^\ell_{\so}\bigr) \Oml\bigl(\bt^\ell_{\st},\bt^\ell_{\so}|\bt^{\ell-1},\bt^{\ell+1}\bigr)\nonumber\\
&- \vk_\ell \Omr\bigl(\bt^\ell_{\so},\bt^\ell_{\st}|\bt^{\ell-1},\bt^{\ell+1}\bigr)\bigr)
\BB\bigl(\bigl\{\bt^s\bigr\}_{1}^{\ell-1},\bt^\ell_{\st},\bigl\{\bt^s\bigr\}_{\ell+1}^{n-1}\bigr) ,\label{Az6}
\end{align}
where the sum in \r{Az6} goes over
partitions $ \bigl\{\bt^\ell_{\so},\bt^\ell_{\st}\bigr\}\vdash \bt^\ell$ such that
$\big|\bt^\ell_{\so}\big|=1$.
The functions $\Oml$ and $\Omr$ in \r{Az6}
are defined as follows:
\begin{gather}
\Oml\bigl(\bt^\ell_{\st},\bt^\ell_{\so}|\bt^{\ell-1},\bt^{\ell+1}\bigr)=\gamma\bigl(\bt^\ell_{\st},\bt^\ell_{\so}\bigr)
 \frac{h\bigl(\bt^\ell_{\so},\bt^{\ell-1}\bigr)}{g\bigl(\bt^{\ell+1},\bt^\ell_{\so}\bigr)},\nonumber\\
\Omr\bigl(\bt^\ell_{\so},\bt^\ell_{\st}|\bt^{\ell-1},\bt^{\ell+1}\bigr)=
\gamma\bigl(\bt^\ell_{\so},\bt^\ell_{\st}\bigr)
 \frac{h\bigl(\bt^{\ell+1},\bt^\ell_{\so}\bigr)}{g\bigl(\bt^\ell_{\so},\bt^{\ell-1}\bigr)} ,\label{Az7}
\end{gather}
where the function $\gamma(u,v)$ is
\begin{equation*}
\gamma(u,v)=\frac{f(u,v)}{h(u,v)h(v,u)}=\frac{g(u,v)}{h(v,u)} .
\end{equation*}

\subsection[Rectangular recurrence relations for gl\_n]{Rectangular recurrence relations for $\boldsymbol{\mathfrak{gl}_n}$}

For $m\in\ZZ$, we define the step function $\Theta(m)$
\[
\Theta(m)=\begin{cases}1,& m\geq 0,\\ 0,&  m<0.
\end{cases}
\]

Theorem~\ref{rec-pr} below will be proved in Appendix~\ref{ApA} by induction starting from \eqref{Az4}. Theorem~\ref{rec-pr} allows to express
off-shell Bethe vectors through the action of the monodromy entries on Bethe
vectors with a smaller number of Bethe parameters. This result is new
although some examples of this type of recurrence relations were know earlier
(see Corollary~\ref{cor1}).

\begin{thm}\label{rec-pr}
For any pair of positive integers $1\leq \ell< k\leq n$, the off-shell
Bethe vector
$\Z_\ell^{k}\cdot \BB\bigl(\bt\bigr)$ satisfies the
{\it rectangular} recurrence relation
\begin{align}
\Z_\ell^{k}\cdot \BB\bigl(\bt\bigr)&=
\BB\bigl(\bigl\{\bt^s\bigr\}_1^{\ell-1},\bigl\{\bu^{s}\bigr\}_\ell^{k-1},\bigl\{\bt^s\bigr\}_k^{n-1}\bigr)\nonumber\\
&=\frac{1}{\mu^k_\ell\bigl(z;\bt\bigr)}
\sum_{i=1}^\ell\sum_{j=k}^n \sum_{\rm part}
 \Xi^{\ell,k}_{i,j}\bigl(z;\bt_{\so},\bt_{\st},\bt_{\sth}\bigr)
T_{i,j}(z)\cdot\BB\bigl(\bt_{\st}\bigr) ,\label{Az19}
\end{align}
where the functions $\Xi^{\ell,k}_{i,j}\bigl(z;\bt_{\so},\bt_{\st},\bt_{\sth}\bigr)$
\begin{align}
\Xi^{\ell,k}_{i,j}\bigl(z;\bt_{\so},\bt_{\st},\bt_{\sth}\bigr)={}&
g\bigl(z, \bt^{\ell-1}_{\so}\bigr) g\bigl(\bt^k_{\sth}, z\bigr)\nonumber\\
&\times
\prod_{s=i}^{\ell-1}\Omr\bigl(\bt^s_{\so},\bt^s_{\st}|\bt^{s-1}_{\st},\bt^{s+1}_{\st}\bigr)
\prod_{s=k}^{j-1}\alpha_s\bigl(\bt^s_{\sth}\bigr)
\Oml\bigl(\bt^s_{\st},\bt^s_{\sth}|\bt^{s-1}_{\st},\bt^{s+1}_{\st}\bigr)\label{A-X-f}
\end{align}
depend on the partitions and
the sum in \r{Az19} goes over partitions
$ \bigl\{\bt^s_{\so},\bt^s_{\st},
\bt^s_{\sth}\bigr\}\vdash \bt^s$
with cardinalities
\begin{equation}\label{A-car}
\big|\bt^s_{\so}\big|=\begin{cases}\Theta(s-i),&s<\ell,\\0,&s\geq \ell,\end{cases}
\qquad
\big|\bt^s_{\sth}\big|=\begin{cases}0,&s\leq k-1,\\
\Theta(j-s-1), &s>k-1.\end{cases}
\end{equation}
 The sets $\bt^\ell,\dots,\bt^{k-1}$ are not partitioned
and in \r{Az19} $\bt^\ell_{\st}=\bt^\ell$ and $\bt^{k-1}_{\st}=\bt^{k-1}$.

If the set $\bt^s = \varnothing$ and according to
\r{A-car} one may have $\big|\bt^s_{\so}\big| = 1$ for some $i$ or
$\big|\bt^s_{\sth}\big| = 1$ for some $j$,
then the terms in the right-hand side
of \r{Az19} for such values of the indices $i$ and $j$ have to be discarded.
\end{thm}

The cardinalities of the subsets $\bt^s_{\so}$ and $\bt^s_{\sth}$ given by
\r{A-car} for $s=1,\dots,n-1$ can be visualized in the following table:
\begin{equation*}
\begin{array}{c|ccc|ccc|ccc|ccc|ccc}
s&1&\cdots&i-1&i&\cdots&\ell-1&\ell&\cdots&k-1&k&\cdots&j-1&j&\cdots&n-1\\[1mm]
\hline\\[-4.6mm]
\big|\bt^s_{\so}\big|&0&\cdots&0&1&\cdots&1&0&\cdots&0&0&\cdots&0&0&\cdots&0\\[1mm]
\big|\bt^s_{\sth}\big|&0&\cdots&0&0&\cdots&0&0&\cdots&0&1&\cdots&1&0&\cdots&0
\end{array}
\end{equation*}

Remark that since $\ell<k$, the cardinalities in \eqref{A-car} show that for any color $s$, either $\bt^s_{\so}$ or $\bt^s_{\sth}$ is empty.
In other words, we always have partitions into a maximum of two subsets,
$ \bigl\{\bt^s_{\so},\bt^s_{\st}\bigr\}\vdash \bt^s$ or
$\bigl\{\bt^s_{\st},\bt^s_{\sth}\bigr\}\vdash \bt^s$
with cardinalities $\big|\bt^s_{\so}\big|\leq1$ and $\big|\bt^s_{\sth}\big|\leq1$.

We call the recurrence relation \r{Az19} a {\it rectangular recurrence relation} because
it contains the monodromy matrix entries $T_{i,j}(z)$ with $1\leq i\leq l$ and $k\leq j\leq n$.
This defines a rectangular submatrix in monodromy matrix with
vertices $T_{1,n}(z)$, $T_{l,n}(z)$, $T_{l,k}$ and $T_{1,k}(z)$. It lies above the diagonal of the monodromy matrix because $l<k$.

\begin{cor}\label{cor1}
Considering the particular case $k=\ell+1$ in the recurrence relation \r{Az19}, we get
\begin{gather*}
\BB\bigl(\bigl\{\bt^s\bigr\}_1^{\ell-1},\bigl\{\bt^\ell,z\bigr\},\bigl\{\bt^s\bigr\}_{\ell+1}^{n-1}\bigr)\\
\qquad=
\sum_{i=1}^\ell\sum_{j=\ell+1}^n\sum_{\rm part}\frac{T_{i,j}(z)\cdot
\BB\bigl(\bigl\{\bt^s\bigr\}_1^{i-1},\bigl\{\bt^s_{\st}\bigr\}_i^{\ell-1},\bt^\ell,
\bigl\{\bt^s_{\st}\bigr\}_{\ell+1}^{j-1},\bigl\{\bt^s\bigr\}_{j}^{n-1}\bigr)}
{\lambda_{\ell+1}(z) g\bigl(z,\bt^{\ell-1}_{\st}\bigr)h\bigl(z,\bt^\ell\bigr)h\bigl(\bt^\ell,z\bigr)g\bigl(\bt^{\ell+1}_{\st},z\bigr)}
\\
\qquad\phantom{=}{}\times
\prod_{s=i}^{\ell-1}\gamma\bigl(\bt^{s}_{\so},\bt^s_{\st}\bigr)
\frac{h\bigl(\bt^{s+1}_{\st},\bt^{s}_{\so}\bigr)}
{g\bigl(\bt^{s}_{\so},\bt^{s-1}_{\st}\bigr)}
\prod_{s=\ell+1}^{j-1}\alpha_s\bigl(\bt^s_{\sth}\bigr)\gamma\bigl(\bt^{s}_{\st},\bt^s_{\sth}\bigr)
\frac{h\bigl(\bt^{s}_{\sth},\bar t^{s-1}_{\st}\bigr)}{g\bigl(\bt^{s+1}_{\st},\bt^{s}_{\sth}\bigr)} .
\end{gather*}
Here the sum over partition is the same as in Theorem~{\rm\ref{rec-pr}}.
\end{cor}
Such type of recurrence relations for arbitrary $\ell$ were written for the first time in \cite{LP22} for the case of integrable
models associated to $U_q(\mathfrak{gl}_n)$. In the case of $\mathfrak{gl}_n$-invariant integrable
models, these recurrence relations for the cases $\ell=1$, $k=2$ and $\ell=n-1$, $k=n$
were presented in~\cite{HLPRS-SP17}.
Analogous recurrence relations for the cases
$\ell=2$, $k=3$ and $\ell=n-2$, $k=n-1$ can be found in \cite{G23}.
Examples of this recurrence relation can also be found in \cite{BPRS13} for the case of $\mathfrak{gl}_3$.

\subsection[Embedding Y(gl\_a)otimes Y(gl\_n-a) Y(gl\_n)]{Embedding $\boldsymbol{Y(\mathfrak{gl}_{a})\otimes Y(\mathfrak{gl}_{n-a})\hookrightarrow Y(\mathfrak{gl}_{n})}$}\label{sect:red-gl.x.gl}

Suppose the Bethe vector does not involve any spectral parameter
of some given color $1\leq a\leq n-1$. In this case, the
 reduced
recurrence relations will imply the following proposition.

\begin{prop}\label{emb-gl}
The off-shell Bethe vectors \smash{$\BBg\bigl(\bigl\{\bt^s\bigr\}_1^{a-1},\varnothing,\bigl\{\bt^s\bigr\}_{a+1}^{n-1}\bigr)$}
has a representation as~the product of $\mathfrak{gl}_{a}$ pre-Bethe vectors
 \smash{$\cBg\bigl(\bigl\{\bt^s\bigr\}_1^{a-1}\bigr) \in Y(\mathfrak{gl}_{a})$} and $\mathfrak{gl}_{n-a}$ pre-Bethe vectors
\linebreak  \smash{$\cBg\bigl(\bigl\{\bt^s\bigr\}_{a+1}^{n-1}\bigr) \in Y(\mathfrak{gl}_{n-a})$} acting on the vacuum vector $\rvec$
 \begin{align}
 \BBg\bigl(\bigl\{\bt^s\bigr\}_1^{a-1},\varnothing,\bigl\{\bt^s\bigr\}_{a+1}^{n-1}\bigr)&=
 \cBg\bigl(\bigl\{\bt^s\bigr\}_1^{a-1}\bigr)\cdot \cBg\bigl(\bigl\{\bt^s\bigr\}_{a+1}^{n-1}\bigr) |0\rangle\nonumber \\
 &=\cBg\bigl(\bigl\{\bt^s\bigr\}_{a+1}^{n-1}\bigr)\cdot
 \cBg\bigl(\bigl\{\bt^s\bigr\}_1^{a-1}\bigr) |0\rangle .\label{emg1}
 \end{align}
\end{prop}
\begin{proof}
We will prove only the first equality in \r{emg1}. The second equality can be proved analogously.
Let us consider the recurrence
relations \r{Az19} for two different ranges of the indices~$\ell$ and $k$:
$1\leq \ell<k\leq a$ and $a+1\leq \ell<k\leq n$.
These ranges for $k$ and $\ell$ ensure that the set~$\bt^a$ remains empty.
Since $\BBg(\varnothing)=\rvec$ (i.e., $\cBg(\varnothing)=1$), the recurrence relations will prove the statement of \eqref{emg1} by induction on the cardinalities.
\begin{itemize}\itemsep=0pt
\item $1\leq \ell<k\leq a$.
The cardinality of the subset $\bt^a_{\sth}$ given by \r{A-car}
is equal to $\Theta(j-a-1)$. Since
this subset is empty, this restricts the summation
over $j$ to $k\leq j\leq a$. This means that the recurrence relation
for $\Z^k_\ell\cdot
\BBg\bigl(\bt\bigr)|_{\bt^a=\varnothing}$ for $1\leq \ell<k\leq a$
implies only the monodromy entries $T_{i,j}(z)$ for $1\leq i\leq \ell$ and
$k\leq j\leq a$. The recurrence relation $\Z^a_{1}\cdot
\BBg\bigl(\bt\bigr)|_{\bt^a=\varnothing}$ will have only one term (there is no summation over
partitions) corresponding to the action of
monodromy entry $T_{1,a}(z)$ on the off-shell Bethe vector.
Then the action of the zero mode operators $\sT_{\ell+1,\ell}$ for
$\ell=1,\dots,a-1$ onto Bethe vector \r{emg1}
\begin{gather*}
\sT_{\ell+1,\ell}\cdot\BBg\bigl(\bigl\{\bt^s\bigr\}_1^{a-1},\varnothing,\bigl\{\bt^s\bigr\}_{a+1}^{n-1}\bigr)\\
\qquad=
\sum_{\rm part}\bigl(\vk_{\ell+1} \alpha_\ell\bigl(\bt^\ell_{\so}\bigr)
\Oml\bigl(\bt^\ell_{\st},\bt^\ell_{\so}|\bt^{\ell-1},\bt^{\ell+1}\bigr)\\
\phantom{\qquad=}{}-\vk_\ell \Omr\bigl(\bt^\ell_{\so},\bt^\ell_{\st}|\bt^{\ell-1},\bt^{\ell+1}\bigr)\bigr)
\BBg\bigl(\bigl\{\bt^s\bigr\}_1^{\ell-1},\bt^\ell_{\st},\bigl\{\bt^s\bigr\}_{\ell+1}^{a-1},
\varnothing,\bigl\{\bt^s\bigr\}_{a+1}^{n-1}\bigr)
\end{gather*}
does not affect the sets of Bethe parameters $\bigl\{\bt^s\bigr\}_{a+1}^{n-1}$
and so coincides with the action of the zero mode operators
onto the Bethe vectors in $\mathfrak{gl}_{a}$-invariant model.
This shows that the recurrence relations for $\Z^k_\ell\cdot
\BBg\bigl(\bt\bigr)|_{\bt^a=\varnothing}$ for $1\leq \ell< k\leq a$
 coincides with the $\mathfrak{gl}_{a}$ recurrence relation for the vector
\smash{$\BBg\bigl(\bigl\{\bt\bigr\}\raisebox{-1pt}{${\vphantom{\big|}}_1^{a-1}$}\bigr)$}.

\item $a+1\leq \ell<k\leq n$. In this case, the cardinality of the empty subset
$\bt^a_{\so}=\varnothing$, given by the equalities
\r{A-car}, yields the restriction for the index $i\geq a+1$.
The summation over $i$ and~$j$
in the right-hand side of the recurrence relation for $\Z^k_\ell\cdot
\BBg\bigl(\bt\bigr)|_{\bt^a=\varnothing}$ will be restricted
to the range $a+1\leq i\leq \ell<k\leq j\leq n$.
The Bethe vector $\Z^n_{a+1}\cdot \BBg\bigl(\bt\bigr)$ will be proportional
to the action of the monodromy matrix entry $T_{a+1,n}(z)\cdot\BBg\bigl(\bt\bigr)$
(there is no sum over partitions in this action) and the
action of the zero mode operators $\sT_{\ell+1,\ell}$ for~${\ell=a+1,\dots,n-1}$ onto the Bethe vector \r{emg1} coincides with
$\mathfrak{gl}_{n-a}$-type actions \r{Az6}.
The recurrence relation~\r{Az19} in this case
will not affect the sets of Bethe parameters \smash{$\bigl\{\bt^s\bigr\}_1^{a-1}$} and will
coincide with the $\mathfrak{gl}_{n-a}$-type
recurrence relation for the Bethe vector \smash{$\Z^k_\ell\cdot\BBg\bigl(\bigl\{\bt^s\bigr\}_{a+1}^{n-1}\bigr)$},
see~\eqref{Az19}.\looseness=-1\hfill \qed
\end{itemize}\renewcommand{\qed}{}
\end{proof}
\begin{rem}
Proposition~\ref{emb-gl} is a direct consequence
 of the projection method, when
 the off-shell Bethe vectors are expressed in terms of the Cartan--Weyl generators of the quantum affine algebra $U_q(\widehat{\mathfrak{gl}}_n)$ {\rm\cite{KhP-Kyoto}} or
 Yangian double $\mathcal{D}Y(\mathfrak{gl}_n)$ {\rm\cite{HLPRS17}}.\end{rem}

{\bf Splitting property of the Bethe vectors.}
Let us define the operation $\Z^{a+1}_a\bigl(t^a_{1}\bigr)$ as the operation which
adds the parameter $t^a_{1}$
of the color $a$
to the set \smash{$\bt=\bigl\{\bigl\{\bt^s\bigr\}^{a-1}_1,\varnothing,\bigl\{\bt^s\bigr\}_{a+1}^{n-1}\bigr\}$} of the Bethe vector \smash{$\BBg\bigl(\bigl\{\bt^s\bigr\}^{a-1}_1,\varnothing,\bigl\{\bt^s\bigr\}_{a+1}^{n-1}\bigr)$}
according to definition \r{Az3}. A direct consequence of
Proposition~\ref{emb-gl} and the rectangular recurrence relations~\r{Az19}
is the presentation of the Bethe vector
\smash{$\BBg\bigl(\bigl\{\bt^s\bigr\}^{a-1}_1,t^a_1,\bigl\{\bt^s\bigr\}_{a+1}^{n-1}\bigr)$} in the form
\begin{align}
\BBg\bigl(\bigl\{\bt^s\bigr\}^{a-1}_1,\{t^a_{1}\},\bigl\{\bt^s\bigr\}_{a+1}^{n-1}\bigr)={}&
\Z^{a+1}_a\bigl(t^a_{1}\bigr)\bigl( \cBg\bigl(\bigl\{\bt^s\bigr\}_1^{a-1}\bigr)\cdot \cBg\bigl(\bigl\{\bt^s\bigr\}_{a+1}^{n-1}\bigr)\bigr) |0\rangle\nonumber\\
={}&\frac{1}{\mu^{a+1}_a\bigl(t^a_1;\bt\bigr)}
\sum_{i=1}^a\sum_{j=a+1}^n \sum_{\rm part}
\Xi^{\ell,k}_{i,j}\bigl(t^a_{1};\bt_{\so},\bt_{\st},\bt_{\sth}\bigr)
T_{i,j}\bigl(t^a_{1}\bigr)\nonumber\\
&\times\bigl( \cBg\bigl(\bigl\{\bt_{\st}^s\bigr\}_1^{a-1}\bigr)\cdot
\cBg\bigl(\bigl\{\bt_{\st}^s\bigr\}_{a+1}^{n-1}\bigr)\bigr) |0\rangle .\label{emg2}
\end{align}
Iterating the relation \r{emg2} one gets the following presentation for
the off-shell Bethe vector in the generic $\mathfrak{gl}_n$-invariant
integrable model:
\begin{equation}\label{emg3}
\BBg\bigl(\bt\bigr)=\Z^{a+1}_a(t^a_1)\cdot \Z^{a+1}_a(t^a_2)\cdots
\Z^{a+1}_a(t^a_{r_a})
\bigl( \cBg\bigl(\bigl\{\bt^s\bigr\}_1^{a-1}\bigr)\cdot \cBg\bigl(\bigl\{\bt^s\bigr\}_{a+1}^{n-1}\bigr)\bigr) |0\rangle .
\end{equation}
The relation \r{emg3} was called in the papers \cite{KT1,KT2} the {\it splitting
property} of the off-shell Bethe vectors, see, e.g., \cite[Proposition 3.1]{KT2}.
One can also compare the relation \r{emg1} with the \cite[equation~(3.9)]{KT2}.
These formulas were proved in these papers
by a different approach using
the evaluation homomorphism for the Yangian
$Y(\mathfrak{gl}_n)$ and the trace formula for the off-shell Bethe vectors.

\begin{exam}\label{Exam31}
Let us consider an example of the recurrence relation \r{emg2} in the
case ${n=4}$ and~${a=2}$. This relation describes the off-shell Bethe vector
$\BBg\bigl(\bt^1,t^2,\bt^3\bigr)$ in the form of the actions of the monodromy
entries $T_{i,j}\bigl(t^2\bigr)$ with $i=1,2$ and $j=3,4$ to the Bethe vector $\BBg\bigl(\bt^1_{\st},\varnothing,\bt^3_{\st}\bigr)$. Let us stress that $t^2$ denotes a single Bethe parameter, not a set.
Using definitions~\r{Az2}
and~\r{A-X-f}, we can write the Bethe vector $\BBg\bigl(\bt^1,t^2,\bt^3\bigr)$
as follows:
\begin{align*}
\BBg\bigl(\bt^1,t^2,\bt^3\bigr)={}&\sum_{i=1}^2\sum_{j=3}^4\sum_{\rm part}
\frac{\alpha_3\bigl(\bt^3_{\sth}\bigr)}{\lambda_3\bigl(t^2\bigr)}
\frac{g\bigl(\bt^1_{\so},\bt^1_{\st}\bigr)}{g\bigl(t^2,\bt^1_{\st}\bigr) h\bigl(\bt^1_{\st},\bt^1_{\so}\bigr)}
\frac{g\bigl(\bt^3_{\st},\bt^3_{\sth}\bigr)}{g\bigl(\bt^3_{\st},t^2\bigr) h\bigl(\bt^3_{\sth},\bt^3_{\st}\bigr)}
 T_{i,j}\bigl(t^2\bigr)\BBg\bigl(\bt^1_{\st},\varnothing,\bt^3_{\st}\bigr) ,
\end{align*}
where sum goes over partitions $\bigl\{\bt^1_{\so},\bt^1_{\st}\bigr\}\vdash\bt^1$
with cardinality $\big|\bt^1_{\so}\big|=\delta_{i,1}$ and
$\bigl\{\bt^3_{\st},\bt^3_{\sth}\bigr\}\vdash\bt^3$
with cardinality $\big|\bt^3_{\sth}\big|=\delta_{j,4}$. In the case when the sets
$\bt^1$ and $\bt^3$ both have cardinality $1$, the Bethe vector
$\BBg\bigl(t^1,t^2,t^3\bigr)$ can be obtained from this recurrence relation and is equal to
\begin{align*}
\begin{split}
\BBg\bigl(t^1,t^2,t^3\bigr)={}&\frac{1}{\lambda_2\bigl(t^1\bigr) \lambda_3\bigl(t^2\bigr) \lambda_4\bigl(t^3\bigr)
g\bigl(t^2,t^1\bigr) g\bigl(t^3,t^2\bigr)}\\
&\times\bigl(T_{2,3}\bigl(t^2\bigr) T_{1,2}\bigl(t^1\bigr) T_{3,4}\bigl(t^3\bigr)+
g\bigl(t^2,t^1\bigr) T_{1,3}\bigl(t^2\bigr) T_{3,4}\bigl(t^3\bigr) T_{2,2}\bigl(t^1\bigr)+
g\bigl(t^3,t^2\bigr)\\
&\times T_{2,4}\bigl(t^2\bigr) T_{1,2}\bigl(t^1\bigr) T_{3,3}\bigl(t^3\bigr)+
g\bigl(t^2,t^1\bigr) g\bigl(t^3,t^2\bigr) T_{1,4}\bigl(t^2\bigr) T_{3,2}\bigl(t^3\bigr) T_{1,1}\bigl(t^1\bigr)\bigr)\rvec .
\end{split}
\end{align*}
Taking into account \r{rtt}, this Bethe vector coincides with the example given at the end of Section~3 in~\cite{KT1} up to normalization and the terms which are annihilated on the vacuum vector~$\rvec$.
\end{exam}

\section[Bethe vectors in o\_2n+1-invariant models]{Bethe vectors in $\boldsymbol{\mathfrak{o}_{2n+1}}$-invariant models}\label{sect:oN}

The off-shell Bethe vectors in $\mathfrak{o}_{2n+1}$-invariant models
depend on the sets of Bethe parameters~${\bt=\bigl(\bt^0,\bt^1,\dots,\bt^{n-1}\bigr)}$
with cardinalities $\big|\bt^s\big|=r_s$ for $s=0,1,\dots,n-1$.

We normalize the $\mathfrak{o}_{2n+1}$-invariant off-shell Bethe vectors in such a way
that the action of the monodromy matrix element $T_{-n,n}(z)$ onto
$\BB\bigl(\bt\bigr)$ is given by the equality
\begin{equation}\label{Bz1}
T_{-n,n}(z)\cdot\BB\bigl(\bt\bigr)= \mu_{-n}^n\bigl(z;\bt\bigr) \BB(\bw) ,
\end{equation}
where
$\bw=\bigl(\bw^0,\bw^1,\dots,\bw^{n-1}\bigr)$ is a collection of extended
sets of Bethe parameters such that
\begin{equation}\label{Bz2}
\bw^s=\bigl\{\bt^s,z,z_s\bigr\},\qquad z_s=z-c(s-1/2)
\end{equation}
and the normalization factor $\mu_{-n}^n\bigl(z;\bt\bigr)$ is given by the expression
\begin{equation}\label{Bz3}
\mu_{-n}^n\bigl(z;\bt\bigr)=-\kappa_n \lambda_n(z)
\frac{g\bigl(z_1,\bar t^0\bigr)}{h\bigl(z,\bar t^{0}\bigr)}
\frac{h\bigl(z,\bar t^{n-1}\bigr)}{g\bigl(z_n,\bar t^{n-1}\bigr)} .
\end{equation}

Besides the sets $\bw^s$ of cardinalities $r_s+2$ which are defined in \r{Bz2},
we will also use the sets~$\bu^s$ and $\bv^s$ of cardinalities $r_s+1$ given by
$
\bu^s=\bigl\{\bt^s,z\bigr\}$, $ \bv^s=\bigl\{\bt^s,z_s\bigr\}$, $ \bw^s=\bigl\{\bu^s,z_s\bigr\}=\bigl\{\bv^s,z\bigr\}$.
For $-n\leq \ell<k\leq n$, let $\Z_\ell^k$ be an operator which extend the sets of
 Bethe parameters depending on the values of $\ell$ and $k$ according to the
rules
\[
\Z_\ell^k\cdot\BB\bigl(\bt\bigr)=\begin{cases}
\BB\bigl(\bigl\{\bt^s\bigr\}_0^{\ell-1},\bigl\{\bu^s\bigr\}_\ell^{k-1},\bigl\{\bt^s\bigr\}_k^{n-1}\bigr),&  0\leq \ell<k\leq n,\\[1mm]
\BB\bigl(\bigl\{\bt^s\bigr\}_0^{-k-1},\bigl\{\bv^s\bigr\}_{-k}^{-\ell-1},\bigl\{\bt^s\bigr\}_{-\ell}^{n-1}\bigr),&  0\leq -k<-\ell\leq n,\\[1mm]
\BB\bigl(\bigl\{\bw^s\bigr\}_0^{-\ell-1},\bigl\{\bu^s\bigr\}_{-\ell}^{k-1},\bigl\{\bt^s\bigr\}_{k}^{n-1}\bigr),& 0\leq -\ell\leq k\leq n,\\[1mm]
\BB\bigl(\bigl\{\bw^s\bigr\}_0^{k-1},\bigl\{\bv^s\bigr\}_{k}^{-\ell-1},\bigl\{\bt^s\bigr\}_{-\ell}^{n-1}\bigr),&  0\leq k\leq- \ell\leq n.
\end{cases}
\]
The action \r{Bz1} of the monodromy matrix entry $T_{-n,n}(z)$
 can be presented as the action
of the operator $\Z_{-n}^n$ onto off-shell Bethe vector
\begin{equation}\label{Bz6}
\Z_{-n}^n\cdot \BB\bigl(\bt\bigr)=
\frac{1}{\mu_{-n}^n\bigl(z;\bt\bigr)}\ T_{-n,n}(z)\cdot\BB\bigl(\bt\bigr) .
\end{equation}

The commutation relations \r{zm1} between the zero mode operators $\sT_{\ell+1,\ell}$ and the
monodromy matrix elements take the form for $0\leq \ell\leq n-1$
\begin{equation}\label{Bz0}
[\sT_{\ell+1,\ell} ,T_{i,j}(z)]=\vk_{\ell+1}
(\delta_{\ell,j-1}-\delta_{\ell,-j}) T_{i,j-1}(z)
-\vk_\ell(\delta_{\ell,i}-\delta_{\ell,-i-1}) T_{i+1,j}(z) .
\end{equation}

The action of the monodromy entry $T_{-n,n}(z)$ \r{Bz1}
as well as the action of the zero mode operators $\sT_{\ell+1,\ell}$
onto off-shell Bethe vectors can be calculated in the framework
 of the projection method (see details in \cite[Section~4.1]{LP2}).
The action of the zero mode operators is{\samepage
\begin{align}
\sT_{\ell+1,\ell}\cdot\BB\bigl(\bt\bigr)={}&\sum_{\rm part}\bigl(\vk_{\ell+1}
\alpha_\ell\bigl(\bt^\ell_{\so}\bigr) \Oml_\ell\bigl(\bt^\ell_{\st},\bt^\ell_{\so}|
\bt^{\ell-1},\bt^{\ell+1}\bigr)\nonumber\\
&- \vk_\ell \Omr_\ell\bigl(\bt^\ell_{\so},\bt^\ell_{\st}|\bt^{\ell-1},\bt^{\ell+1}\bigr)\bigr)
\BB\bigl(\bigl\{\bt^s\bigr\}_1^{\ell-1},\bt^\ell_{\st},\bigl\{\bt^s\bigr\}_{\ell+1}^{n-1}\bigr) ,\label{Bz8}
\end{align}}%

\pagebreak

\noindent
where the sum is over partitions with cardinalities $\big|\bt^\ell_{\so}\big|=1$.
In \r{Bz8}, the rational functions
$\Oml_\ell$ and
$\Omr_\ell$ have
expressions similar to \r{Az7}
\begin{gather}
\Oml_\ell\bigl(\bt^\ell_{\st},\bt^\ell_{\so}|\bt^{\ell-1},\bt^{\ell+1}\bigr)=
\gamma_\ell\bigl(\bt^\ell_{\st},\bt^\ell_{\so}\bigr)
 \frac{h\bigl(\bt^\ell_{\so},\bt^{\ell-1}\bigr)}{g\bigl(\bt^{\ell+1},\bt^\ell_{\so}\bigr)},\nonumber\\
\Omr_\ell\bigl(\bt^\ell_{\so},\bt^\ell_{\st}|\bt^{\ell-1},\bt^{\ell+1}\bigr)=
\gamma_\ell\bigl(\bt^\ell_{\so},\bt^\ell_{\st}\bigr)
 \frac{h\bigl(\bt^{\ell+1},\bt^\ell_{\so}\bigr)}{g\bigl(\bt^\ell_{\so},\bt^{\ell-1}\bigr)} ,\label{Bz10}
\end{gather}
where
\[
\gamma_\ell(u,v)=
\begin{cases}
\fgo(u,v)=\dfrac{u-v+c/2}{u-v},&\ell=0,\\
\gamma(u,v)=\dfrac{g(u,v)}{h(v,u)}=\dfrac{c^2}{(u-v)(v-u+c)},
&\ell=1,\dots,n-1.
\end{cases}
\]

\subsection[Rectangular recurrence relations for o\_2n+1]{Rectangular recurrence relations for $\boldsymbol{\mathfrak{o}_{2n+1}}$}\label{sect:rec-oN}

The main result of the paper is formulated in Theorem~\ref{main-th} below.
The proof of this theorem is similar to the proof of Theorem~\ref{rec-pr} given in Appendix~\ref{ApA}. In Appendix \ref{ApB}, we will sketch its proof.

For the formulation of the statement of the theorem, we need following
notations:
\begin{itemize}\itemsep=0pt
\item the sign factor $\sigma_m$
\[
\sigma_m=2\Theta(m-1)-1=\begin{cases}
1,  &m>0,\\
-1,  &m\leq 0,
\end{cases}
\]
which satisfy the property
$
\sigma_{m+1}=-\sigma_{-m}$,
\item the functions $\psi_\ell\bigl(z;\bt\bigr)$, $\phi_k\bigl(z;\bt\bigr)$, $\mu^k_\ell\bigl(z;\bt\bigr)$
\begin{gather}
\psi_\ell\bigl(z;\bt\bigr)=\begin{cases}
g\bigl(z,\bt^{\ell-1}\bigr)\ h\bigl(\bt^\ell,z\bigr), &0<\ell<n,\\
g\bigl(z_0,\bt^0\bigr), &\ell=0,\\
\dfrac{g(\bt^{-\ell},z_{-\ell})}{g(z_{-\ell},\bt^{-\ell-1})}, &-n\leq \ell<0,
\end{cases}\label{psi}
\\
\phi_k\bigl(z;\bt\bigr)=\begin{cases}
h\bigl(z,\bt^{k-1}\bigr)\ g\bigl(\bt^k,z\bigr), &0<k\leq n,\\
g(z,\bt^0), &k=0,\\
\dfrac{g(z_{-k-1},\bt^{-k-1})}{g(\bt^{-k},z_{-k-1})}, &-n< k<0,\label{phi}
\end{cases}
\\
\mu^k_\ell\bigl(z;\bt\bigr)=\sigma_{-\ell-k}
(\kappa_k)^{\delta_{k,-\ell}} \lambda_k(z) \psi_\ell\bigl(z;\bt\bigr) \phi_k\bigl(z;\bt\bigr)
\left(\dfrac{g\bigl(z_1,\bt^0\bigr)}{h\bigl(z,\bt^0\bigr)}\right)^{\delta_{\ell<0, k>0}} ,\nonumber
\end{gather}
where $\kappa_k=k-1/2$ and $\delta_{\rm condition}$ is equal to 1 if
 "condition" is satisfied and to 0 otherwise.
\item the partitions
\begin{gather}
\big|\bt^s_{\so}\big|=\begin{cases}
\Theta(\ell)\bigl(\Theta(s-i)+\Theta(-i-s-1)\bigr),& s<|\ell|,\\
\Theta(-i-s-1),& s\geq |\ell|,\label{part-so}
\end{cases}
\\
\big|\bt^s_{\sth}\big|=\begin{cases}
\Theta(-k)(\Theta(j+s)+\Theta(j-s-1)),& s<|k|,\\
\Theta(j-s-1),& s\geq |k|,
\end{cases}\label{part-sth}
\\
\big|\bt^s_{\st}\big|=|\bt^s|-\big|\bt^s_{\so}\big|-\big|\bt^s_{\sth}\big| ,\label{part-st}
\end{gather}
\item the functions $\Gml_{a,b}\bigl(\bt_{\st},\bt_{\so}\bigr)$
and $\Gmr_{a,b}\bigl(\bt_{\so},\bt_{\st}\bigr)$
for $a,b=0,1,\dots,n$
\textit{which depend on the partitions}
\begin{gather*}
\Gml_{a,b}\bigl(\bt_{\st},\bt_{\so}\bigr)
=\prod_{s=a}^{b-1}\alpha_s\bigl(\bt^s_{\so}\bigr)\Oml_s\bigl(\bt^s_{\st},\bt^s_{\so}|\bt^{s-1}_{\st},\bt^{s+1}_{\st}\bigr),\\
\Gmr_{a,b}\bigl(\bt_{\so},\bt_{\st}\bigr)
=\prod_{s=a}^{b-1}\Omr_s\bigl(\bt^s_{\so},\bt^s_{\st}|\bt^{s-1}_{\st},\bt^{s+1}_{\st}\bigr),
\end{gather*}
and the functions $\Xi^{\ell,k}_{i,j}\bigl(z;\bt_{\so},\bt_{\st},\bt_{\sth}\bigr)$
\begin{align}
\Xi^{\ell,k}_{i,j}\bigl(z;\bt_{\so},\bt_{\st},\bt_{\sth}\bigr)&=
\psi_\ell\bigl(z,\bt_{\so}\bigr) \phi_k\bigl(z,\bt_{\sth}\bigr) \Gmr_{0,n}\bigl(\bt_{\so},\bt_{{\st},{\sth}}\bigr)
\Gml_{0,n}\bigl(\bt_{\st},\bt_{{\sth}}\bigr)\nonumber\\
& =
\psi_\ell\bigl(z,\bt_{\so}\bigr) \phi_k\bigl(z,\bt_{\sth}\bigr)
\Gmr_{0,n}\bigl(\bt_{\so},\bt_{{\st}}\bigr)
\Gml_{0,n}\bigl(\bt_{{\so},{\st}},\bt_{{\sth}}\bigr) ,\label{X-fun}
\end{align}
where the equality between the first and second lines in \r{X-fun}
directly follows from the definitions \r{Bz10},
with the notation $\bt_{\so,\st} = \bigl\{\bt_{\so}, \bt_{\st}\bigr\}$, and $\bt_{\st,\sth} = \bigl\{\bt_{\st}, \bt_{\sth}\bigr\}$.
The right hand side of \eqref{X-fun} depends on $i$ and $j$ through the cardinalities of the subsets $\bt_{\so}$, $\bt_{{\st}}$, $\bt_{{\sth}}$, see equations~\eqref{part-so} and \eqref{part-sth}.
Note that $\Gml_{a,b}\bigl(\bt_{\st},\bt_{\so}\bigr)=\Gmr_{a,b}\bigl(\bt_{\so},\bt_{\st}\bigr) = 1$
for $a\geq b$.
\end{itemize}

\begin{thm}\label{main-th}
For any $\ell$, $k$ such that $-n\leq \ell <k\leq n$, the off-shell Bethe vector
$\Z^k_\ell\cdot\BB\bigl(\bt\bigr)$ satisfies the rectangular recurrence relation
\begin{align}
\Z^k_\ell\cdot\BB\bigl(\bt\bigr)={}&\frac{1}{\mu^k_\ell\bigl(z;\bt\bigr)}\sum_{i=-n}^\ell
\sum_{j=k}^n\sum_{\rm part}
\Xi^{\ell,k}_{i,j}\bigl(z;\bt_{\so},\bt_{\st},\bt_{\sth}\bigr) (\sigma_{-i})^{\delta_{\ell\geq0}}
(\sigma_{j})^{\delta_{k\leq 0}} T_{i,j}(z)\cdot\BB\bigl(\bt_{\st}\bigr) ,\label{main-rr}
\end{align}
where the sum is over partitions
$\bigl\{\bt^s_{\so},\bt^s_{\st},\bt^s_{\sth}\bigr\}\vdash \bt^s$ with
cardinalities depending on the indices $-n\leq i\leq \ell<k\leq j\leq n$ and given by the equalities \r{part-so}, \r{part-sth} and \r{part-st}.

If the cardinality of the set $\bt^s$ is small and
according to \r{part-so} and \r{part-sth}, one may have $\big|\bt^s_{\so}\big|+\big|\bt^s_{\sth}\big|>|\bt^s|$
for some $i$ and $j$. Then the terms in
the right-hand side of \r{main-rr} for such values of the indices $i$ and $j$ have to be discarded.
\end{thm}

Let us stress that, in opposition to the $\mathfrak{gl}_n$ case, for
$\mathfrak{o}_{2n+1}$-invariant models, depending on the color $s$,
the subsets $\bt^s_{\so}$ and $\bt^s_{\sth}$ satisfy $\big|\bt^s_{\so}\big|+\big|\bt^s_{\sth}\big| \le 2$.
Theorem~\ref{main-th} can be proved using an inductive
approach similar to the one used to prove Theorem~\ref{rec-pr} in Appendix~\ref{ApA}. Sketch of the proof is given in Appendix~\ref{ApB}.
This result is new, except for two partial cases presented in~\cite{LP2} when $\ell=n-1$.

\subsection{Special cases of recurrence relations}\label{sect:ex-rec}
We provide subcases of the recurrence relations
\r{main-rr} that are relevant for the study of integrable models.

{\bf Elementary recurrence relations.}
For explicit calculations, the most interesting (and simplest) cases of the recurrence
relations \r{main-rr} are the $2n$ cases when $k=\ell+1$ and $\ell=-n,\dots,n-1$.
These $2n$ cases are gathered in four classes of expressions, depending on the value of the index $\ell$.
\begin{itemize}\itemsep=0pt
\item When $0<\ell<n$, the recurrence relation reads
\begin{align*}
\begin{split}
\Z^{\ell+1}_{\ell}\cdot\BB\bigl(\bt\bigr)={}&
\BB\bigl(\bt^0,\dots,\bt^{\ell-1},\bigl\{\bt^\ell,z\bigr\},\bt^{\ell+1},\dots,\bt^{n-1}\bigr)\\
={}&\sum_{i=-n}^{\ell}\sum_{j=\ell+1}^n\sum_{\rm part}
\frac{g\bigl(z,\bt^{\ell-1}_{\so}\bigr) h\bigl(\bt^{\ell}_{\so},z\bigr) g\bigl(\bt^{\ell+1}_{\sth},z\bigr)}
{\lambda_{\ell+1}(z) g\bigl(z,\bt^{\ell-1}\bigr) h\bigl(\bt^{\ell},z\bigr)
h\bigl(z,\bt^{\ell}\bigr) g\bigl(\bt^{\ell+1},z\bigr)}\\
&\times \Gmr_{[i],\ell_i}\bigl(\bt_{\so},\bt_{{\st}}\bigr)
\Gml_{\ell+1,j}\bigl(\bt_{{\so},{\st}},\bt_{{\sth}}\bigr) \sigma_{i+1}
T_{i,j}(z)\cdot \BB\bigl(\bt_{\st}\bigr) ,
\end{split}
\end{align*}
where $\ell_i={\rm max}(\ell,|i|)$, $[i]=\frac12(i+|i|)$, and the cardinalities of the subsets
\[
\big|\bt^s_{\so}\big|=\begin{cases}
\Theta(s-i)+\Theta(-i-s-1), &s<\ell,\\
\Theta(-i-s-1), &s\geq \ell,
\end{cases}\qquad
\big|\bt^s_{\sth}\big|=\begin{cases}
0, &s<\ell+1,\\
\Theta(j-s-1), &s\geq \ell+1.
\end{cases}
\]
Remark that for $\ell=n-1$ this recurrence relation reduces to
\begin{align*}
\BB\bigl(\bigl\{\bt^s\bigr\}_{0}^{n-2},\bigl\{\bt^{n-1},z\bigr\}\bigr)={}&\frac{1}{\lambda_n(z) h\bigl(z,\bt^{n-1}\bigr)}\sum_{i=-n}^{n-1}\sum_{\rm part}
\frac{\sigma_{i+1} T_{i,n}(z)\cdot
\BB\bigl(\bt_{\st}\bigr)}{g\bigl(z,\bt^{n-2}_{\st}\bigr) h\bigl(\bt^{n-1}_{\st},z\bigr)}\\
&\times
\prod_{s=0}^{n-1}\gamma_s\bigl(\bt^s_{\so},\bt^s_{\st}\bigr)
\prod_{s=1}^{n-1}\frac{h\bigl(\bt^s_{\st},\bt^{s-1}_{\so}\bigr)}{g\bigl(\bt^s_{\so},\bt^{s-1}_{\st}\bigr)},
\end{align*}
which coincides with the
recurrence relation (3.16) from the paper \cite{LP2}.
\item The cases of $-n\leq \ell<-1$ corresponds to so called {\it shifted}
 recurrence
relation for the off-shell Bethe vectors.
To describe these cases, we set
$\ell=-l-1$ with $0<l<n$. They correspond to an extension of the set of Bethe
parameters $\bt^l$ by the shifted parameter~${z_l=z-c(l-1/2)}$
\begin{align}
\Z^{-l}_{-l-1}\cdot\BB\bigl(\bt\bigr)={}&
\BB\bigl(\bt^0,\dots,\bt^{l-1},\bigl\{\bt^{l},z_{l}\bigr\},\bt^{l+1},\dots,\bt^{n-1}\bigr)\nonumber\\
={}&\sum_{i=-n}^{-l-1}\sum_{j=-l}^n\sum_{\rm part}
\frac{(-1)^{\delta_{j\geq l+1}}
g\bigl(\bt^{l+1}_{\so},z_{l+1}\bigr) h\bigl(z_l,\bt^{l}_{\sth}\bigr) g\bigl(z_{l-1},\bt^{l-1}_{\sth}\bigr)}
{\lambda_{-l}(z) g\bigl(\bt^{l+1},z_{l+1}\bigr) h\bigl(\bt^{l},z_{l}\bigr) h\bigl(z_l,\bt^{l}\bigr)
g\bigl(z_{l-1},\bt^{l-1}\bigr)}\nonumber\\
&\times \Gmr_{l+1,-i}\bigl(\bt_{\so},\bt_{{\st},{\sth}}\bigr)
\Gml_{[-j],l_j}\bigl(\bt_{\st},\bt_{{\sth}}\bigr)\sigma_{j}
T_{i,j}(z)\cdot \BB\bigl(\bt_{\st}\bigr) ,\label{ex2}
\end{align}
where $l_j={\rm max}(l,|j|)$, $[j]=\frac12(j+|j|)$ as above, and the
 cardinalities of the subsets in the sum over partitions are
\begin{gather}
\big|\bt^s_{\so}\big|=\begin{cases}
0, &s<l+1,\\
\Theta(-i-s-1), &s\geq l+1,
\end{cases}\nonumber\\
\big|\bt^s_{\sth}\big|=\begin{cases}
\Theta(s+j)+\Theta(j-s-1), &s<l,\\
\Theta(j-s-1), &s\geq l.
\end{cases}\label{ex2a}
\end{gather}
The sign factor $(-1)^{\delta_{j\geq l+1}}$ in the second line of
\r{ex2} is due to the identity
 $g\bigl(\bt^l_{\sth},z_{l-1}\bigr)^{-1}=-h\bigl(z_{l},\bt^l_{\sth}\bigr)$, valid when the cardinality of the subset
$\bt^l_{\sth}$ is equal to one. According to~\r{ex2a}, this happens
when $j\geq l+1$.

For $l=n-1$, the recurrence relation \r{ex2} can be rewritten in the form
\begin{gather*}
 \BB\bigl(\bigl\{\bt^s\bigr\}_{0}^{n-2},\bigl\{\bt^{n-1},z_{n-1}\bigr\}\bigr)
 \\
\qquad = \frac{1}{\lambda_{-n+1}(z) h\bigl(\bt^{n-1},z_{n-1}\bigr)} \sum_{j=-n+1}^n\sum_{\rm part}
 \frac{(-1)^{\delta_{j,n}}}{h\bigl(z_{n-1},\bt^{n-1}_{\st}\bigr)}
 \frac{\sigma_{j} T_{-n,j}(z)\cdot\BB\bigl(\bt_{\st}\bigr)}{g\bigl(z_{n-2},\bt^{n-2}_{\st}\bigr)}\\
 \qquad\phantom{=}{}\times
 \prod_{s=0}^{n-1}\alpha_s\bigl(\bt^s_{\sth}\bigr)\gamma_s\bigl(\bt^s_{\st},\bt^s_{\sth}\bigr)
 \prod_{s=1}^{n-1}\frac{h\bigl(\bt^s_{\sth},\bt^{s-1}_{\st}\bigr)}{g\bigl(\bt^s_{\st},\bt^{s-1}_{\sth}\bigr)}
 \end{gather*}
and
coincides exactly with the
recurrence relation (3.19) from the paper \cite{LP2}. The sign factor $(-1)^{\delta_{j\geq l+1}}$
turns in this case to the sign factor $(-1)^{\delta_{j,n}}$.

\item When $\ell=0$, the recurrence relation takes the form
\begin{align}
\Z^{1}_{0}\cdot\BB\bigl(\bt\bigr)={}&\BB\bigl(\bigl\{\bt^0,z\bigr\},\bt^{1},\dots,\bt^{n-1}\bigr)\nonumber\\
={}&\sum_{i=-n}^{0}\sum_{j=1}^n\sum_{\rm part}
\frac{g\bigl(z_0,\bt^{0}_{\so}\bigr) g\bigl(\bt^{1}_{\sth},z\bigr) \sigma_{i+1}
T_{i,j}(z)\cdot \BB\bigl(\bt_{\st}\bigr)}
{\lambda_{1}(z) g\bigl(z_0,\bt^{0}\bigr) h\bigl(z,\bt^{0}\bigr)
g\bigl(\bt^{1},z\bigr)}\nonumber\\
&\times
\prod_{s=0}^{-i-1}\gamma_s\bigl(\bt^s_{\so},\bt^s_{\st}\bigr)
\frac{h\bigl(\bt^{s+1}_{\st},\bt^{s}_{\so}\bigr)}{g\bigl(\bt^s_{\so},\bt^{s-1}_{\st}\bigr)}
 \prod_{s=1}^{j-1}\alpha_s\bigl(\bt^s_{\sth}\bigr)\gamma_s\bigl(\bt^s_{{\so},{\st}},\bt^s_{\sth}\bigr)
 \frac{h\bigl(\bt^s_{\sth},\bt^{s-1}_{{\so},{\st}}\bigr)}
 {g\bigl(\bt^{s+1}_{{\so},{\st}},\bt^{s}_{\sth}\bigr)} ,\label{ex3}
\end{align}
where cardinalities of the subsets in the sum over partitions are
\[
\big|\bt^s_{\so}\big|=\Theta(-i-s-1)\quad \forall s,
\qquad
\big|\bt^s_{\sth}\big|=\begin{cases}
0, &s=0,\\
\Theta(j-s-1), &s\geq 1.
\end{cases}
\]
\item When $\ell=-1$, one gets the recurrence
relation
\begin{align}
\Z^{0}_{-1}\cdot\BB\bigl(\bt\bigr)={}&\BB\bigl(\bigl\{\bt^0,z_0\bigr\},\bt^{1},\dots,\bt^{n-1}\bigr)\nonumber\\
={}&\sum_{i=-n}^{-1}\sum_{j=0}^n\sum_{\rm part}
\frac{g\bigl(\bt^{1}_{\so},z_1\bigr) g\bigl(z,\bt^{0}_{\sth}\bigr) \sigma_j T_{i,j}(z)\cdot \BB\bigl(\bt_{\st}\bigr)}
{\lambda_{0}(z) g\bigl(\bt^{1},z_1\bigr) g\bigl(z_1,\bt^{0}\bigr)^{-1}
g\bigl(z,\bt^{0}\bigr)}\nonumber\\
&\times
\prod_{s=1}^{-i-1}\gamma_s\bigl(\bt^s_{\so},\bt^s_{{\st},{\sth}}\bigr)
\frac{h\bigl(\bt^{s+1}_{{\st},{\sth}},\bt^{s}_{\so}\bigr)}
{g\bigl(\bt^s_{\so},\bt^{s-1}_{{\st},{\sth}}\bigr)}
 \prod_{s=0}^{j-1}\alpha_s\bigl(\bt^s_{\sth}\bigr)\gamma_s\bigl(\bt^s_{{\st}},\bt^s_{\sth}\bigr)
 \frac{h\bigl(\bt^s_{\sth},\bt^{s-1}_{{\st}}\bigr)}{g\bigl(\bt^{s+1}_{{\st}},\bt^{s}_{\sth}\bigr)} ,\label{ex4}
\end{align}
where cardinalities of the subsets in the sum over partitions are
\[
\big|\bt^s_{\so}\big|=\begin{cases}
0,&s=0,\\
\Theta(-i-s-1), &s\geq 1,
\end{cases}
\qquad
\big|\bt^s_{\sth}\big|=
\Theta(j-s-1)\quad \forall s\geq 1.
\]
\end{itemize}

{\bf Case of $\boldsymbol{\mathfrak{o}_{3}}$-invariant models.}
The above recurrence relations were calculated for the algebras $\mathfrak{o}_{2n+1}$
with $n>1$. However, one can verify that they are also valid for
the orthogonal algebra $\mathfrak{o}_3$ when $n=1$.
This case was investigated in details in the paper \cite{LPRS19}.
In this case, there are only three cases of rectangular recurrence relations
for the Bethe vectors:
$\Z^{1}_{-1}\cdot\BB\bigl(\bt^0\bigr)$, $\Z^{1}_{0}\cdot\BB\bigl(\bt^0\bigr)$ and
$\Z^{0}_{-1}\cdot\BB\bigl(\bt^0\bigr)$, that we detail below.

For the Bethe vector $\Z^{1}_{-1}\cdot\BB\bigl(\bt^0\bigr)$, there is no summation over partitions
 and according to~\r{Bz1} and \r{Bz3}
it is equal to
\begin{equation}\label{B1-1}
\Z^{1}_{-1}\cdot\BB\bigl(\bt^0\bigr)=
\BB\bigl(\bigl\{\bt^0,z,z_0\bigr\}\bigr)= -\frac{2}{\lambda_1(z)} T_{-1,1}(z)\cdot \BB\bigl(\bt^0\bigr) .
\end{equation}
For the Bethe vector $\Z^{1}_{0}\cdot\BB\bigl(\bt^0\bigr)$, the recurrence
relation \r{ex3} becomes
\begin{align}
\Z^{1}_{0}\cdot\BB\bigl(\bt^0\bigr)&=\BB\bigl(\bigl\{\bt^0,z\bigr\}\bigr)\nonumber\\
 &=\frac{1}{\lambda_1(z)\fgo\bigl(z_0,\bt^0\bigr)}\!
\sk{T_{0,1}(z)\cdot\BB\bigl(\bt^0\bigr)-\!\sum_{\rm part}
g\bigl(z_0,\bt^0_{\so}\bigr)\fgo\bigl(\bt^0_{\so},\bt^0_{\st}\bigr) T_{-1,1}(z)\cdot\BB\bigl(\bt^0_{\st}\bigr)} ,\label{B1-2}
\end{align}
while for the Bethe vector $\Z^{0}_{-1}\cdot\BB\bigl(\bt^0\bigr)$
the recurrence relation \r{ex4} becomes
\begin{align}
 \Z^{0}_{-1}\cdot\BB\bigl(\bt^0\bigr)={}&\BB\bigl(\bigl\{\bt^0,z_0\bigr\}\bigr)=\frac{1}{\lambda_0(z)\fgo\bigl(\bt^0,z\bigr)}\nonumber\\
 &\times\!
\sk{-T_{-1,0}(z)\cdot\BB\bigl(\bt^0\bigr)+\sum_{\rm part}
g\bigl(z,\bt^0_{\sth}\bigr) \alpha_0\bigl(\bt^0_{\sth}\bigr) \fgo\bigl(\bt^0_{\st},\bt^0_{\sth}\bigr)
T_{-1,1}(z)\cdot\BB\bigl(\bt^0_{\st}\bigr)\!} .\label{B1-3}
\end{align}
In \r{B1-2} and \r{B1-3}, we have $\big|\bt^0_{\so}\big|=|\bt^0_{\sth}|=1$.
Formulas \r{B1-1}, \r{B1-2} and \r{B1-3} coincide with
\cite[formulas (4.15), (4,19) and (4.20)]{LPRS19}. This proves that
Theorem~\ref{main-th} is also valid for $\mathfrak{o}_3$-invariant integrable models.

\subsection[Reduction to gl\_n]{Reduction to $\boldsymbol{\mathfrak{gl}_n}$}\label{sect:red-gln}
Still in the framework of the $\mathfrak{o}_{2n+1}$-type monodromy matrix,
we consider the particular case of Bethe vectors with $\bt^0=\varnothing$. In that case,
since the color 0 is absent, one should recover the construction done for the
$\mathfrak{gl}_{n}$-type monodromy matrix. In this subsection, we prove it by showing that the recurrence relations
\r{main-rr}, coming from the $\mathfrak{o}_{2n+1}$ set-up, provide the recurrence relations \eqref{Az19},
obtained in the
$\mathfrak{gl}_{n}$ case. Obviously, to stay within the $\mathfrak{gl}_{n}$ part, we need to keep
$\bt^0=\varnothing$. Thus, we only consider the recurrence relations
\r{main-rr} with either $-n\leq \ell< k<0$ or $0<\ell< k\leq n$.

{\bf The case $\boldsymbol{0<\ell<k\leq n}$ and $\boldsymbol{\bt^0=\varnothing}$.} According to the first line in
\r{part-so}, it signifies that~${\big|\bt^0_{\so}\big|=0=\Theta(-i)+\Theta(-i-1)}$
which is possible only for $i>0$. Since for these values of the index $i$
the step function
$\Theta(-i-s-1)$ vanishes for any $s$, the cardinalities of the partitions
\r{part-so} and \r{part-sth} become the cardinalities \r{A-car}. Moreover, for $i>0$,
 the sign $\delta_{-\ell-k}$ compensate the sign factor
$\sigma_{-i}=-1$ and $[i]=i$ and $\ell_i=\ell$, so that
the function $\Xi^{\ell,k}_{i,j}\bigl(z;\bt_{\so},\bt_{\st},\bt_{\sth}\bigr)$ in \r{X-fun}
becomes the function
\r{A-X-f}.
The recurrence relation \r{main-rr} takes the form
\begin{gather}
\BB\bigl(\varnothing,\bigl\{\bt^s\bigr\}_1^{\ell-1},\bigl\{\bt^{s},z\bigr\}_{\ell}^{k-1},\bigl\{\bt^s\bigr\}_k^{n-1}\bigr)\nonumber\\
\qquad=\sum_{i=1}^{\ell}\sum_{j=k}^n \sum_{\rm part}
\frac{\TT_{i,j}(z)\cdot\BB\bigl(\varnothing,\bigl\{\bt^s\bigr\}_{1}^{i-1},\bigl\{\bt^s_{\st}\bigr\}_{i}^{\ell-1},
\bigl\{\bt^s\bigr\}_{\ell}^{k-1},\bigl\{\bt^s_{\st}\bigr\}_{k}^{j-1},\bigl\{\bt^s\bigr\}_{j}^{n-1}\bigr)}
{\lambda_k(z) g\bigl(z,\bt^{\ell-1}_{\st}\bigr) h\bigl(\bt^{\ell},z\bigr) h\bigl(z,\bt^{k-1}\bigr)
g\bigl(\bt^k_{\st},z\bigr)}\nonumber\\
\phantom{\qquad=}{}\times
\prod_{s=i}^{\ell-1}\Omr\bigl(\bt^s_{\so},\bt^s_{\st}|\bt^{s-1}_{\st},\bt^{s+1}_{\st}\bigr)
\prod_{s=k}^{j-1}\alpha_s\bigl(\bt^s_{\so}\bigr)
\Oml\bigl(\bt^s_{\st},\bt^s_{\so}|\bt^{s-1}_{\st},\bt^{s+1}_{\st}\bigr) ,\label{gl-red1}
\end{gather}
where the cardinalities of the subsets in the sum over partitions
are given by \r{A-car} and $\TT_{i,j}(z):= T_{i,j}(z)$ for
$1\leq i,j\leq n$ as it was described for the first type of
embedding of $Y(\mathfrak{gl}_n)$ in $Y(\mathfrak{o}_{2n+1})$, see~\r{Tgl}. So we recover the
$\mathfrak{gl}_n$-type recurrence relations \r{Az19}.

{
\bf The case $\boldsymbol{-n\leq \ell<k<0}$ and $\boldsymbol{\bt^0=\varnothing}$.} This case is more subtle
because it is related to the off-shell Bethe vector
$\hBB\bigl(\bt\bigr)$, connected to $\BB\bigl(\bt\bigr)$
in the sense of the paper \cite{LPRS-New}.
If
the off-shell Bethe vector $\BB\bigl(\bt\bigr)$ is built from the entries
$\TT_{i,j}(z)=T_{i,j}(z)$ for $1\leq i<j\leq n$ of $\mathfrak{gl}_n$-type
according to the embedding \r{Tgl},
the off-shell Bethe vector \small{$\hBB\bigl(\bt\bigr)$} is built from the
$\mathfrak{gl}_n$-type matrix entries~$\hat{\TT}_{i,j}(z)$ described
by the embedding \r{em-sh}.

Let us consider the recurrence
relations \r{main-rr} at $-n\leq \ell<k<0$ and empty set $\bt^0=\varnothing$ in
more details.
If $\bt^0=\varnothing$, then according to the
first line in \r{part-sth} $\big|\bt^0_{\sth}\big|=0=\Theta(j)+\Theta(j-1)$ which is possible
only for strictly negative values of the index $j$. So the recurrence
relation \r{main-rr} becomes
\begin{align}
\Z^k_\ell\cdot\BB\bigl(\bt\bigr)\big|_{\bt^0=\varnothing}&=
\BB\bigl(\varnothing,\bigl\{\bt^s\bigr\}_1^{-k-1},\bigl\{\bt^s,z_s\bigr\}_{-k}^{-\ell-1},\bigl\{\bt^s\bigr\}_{-\ell}^{n-1}\bigr)\nonumber\\
&= - \mu^k_\ell\bigl(z;\bt\bigr)^{-1} \sum_{i=-n}^\ell
\sum_{j=k}^{-1}\sum_{\rm part}
\Xi^{\ell,k}_{i,j}\bigl(z;\bt_{\so},\bt_{\st},\bt_{\sth}\bigr)
T_{i,j}(z)\cdot\BB\bigl(\varnothing,\bigl\{\bt^s_{\st}\bigr\}_1^{n-1}\bigr) ,\label{gl1}
\end{align}
where $\mu^k_\ell\bigl(z;\bt\bigr)=\lambda_k(z) \psi_\ell\bigl(z;\bt\bigr) \phi_k\bigl(z;\bt\bigr)$ and
\begin{align}
\Xi^{\ell,k}_{i,j}\bigl(z;\bt_{\so},\bt_{\st},\bt_{\sth}\bigr)={}&
\psi_\ell\bigl(z;\bt_{\so}\bigr)\ \phi_k\bigl(z;\bt_{\sth}\bigr)\nonumber\\
&\times
\prod_{s=-\ell}^{-i-1}\Omr\bigl(\bt^s_{\so},\bt^s_{{\st}}|\bt^{s-1}_{\st},\bt^{s+1}_{\st}\bigr)
\prod_{s=-j}^{-k-1}\alpha_s\bigl(\bt^s_{\sth}\bigr)
\Oml\bigl(\bt_{\st},\bt_{{\sth}}|\bt^{s-1}_{\sth},\bt^{s+1}_{\st}\bigr)\label{X-red}
\end{align}
since $\bt^s_{\sth}=\varnothing$ for $-\ell\leq s\leq -i-1$.
 In \r{X-red}, we dropped the index $s$ of the functions
$\Oml_s$ and $\Omr_s$ because for positive $s$ these functions
does not depend on $s$ and coincide with the functions \r{Az7}.

The functions
$\psi_\ell\bigl(z;\bt\bigr)$ and $\phi_k\bigl(z;\bt\bigr)$ are
\[
\psi_\ell\bigl(z;\bt\bigr)=\frac{g\bigl(\bt^{-\ell},z_{-\ell}\bigr)}{g\bigl(z_{-\ell},\bt^{-\ell-1}\bigr)},\qquad
\phi_k\bigl(z;\bt\bigr)=\frac{g\bigl(z_{-k-1},\bt^{-k-1}\bigr)}{g\bigl(\bt^{-k},z_{-k-1}\bigr)}
\]
and the cardinalities of the partitions in \r{gl1} are
\begin{equation}\label{gl3}
\big|\bt^s_{\so}\big|=\begin{cases}
0, &s<-\ell,\\
\Theta(-i-s-1), &s\geq -\ell,
\end{cases}\qquad
\big|\bt^s_{\sth}\big|=\begin{cases}
\Theta(j+s), &s<-k,\\
0, &s\geq -k.
\end{cases}
\end{equation}
Note again that there are no summations in the right-hand side of \r{gl1}
in case when $\ell=-n$ and $k=-1$.

Let us perform the following transformation of the recurrence relation
\r{gl1}. First we replace the strictly negative indices $-n\leq \ell<k<0$
by the strictly positive indices $0<\lp<\kp\leq n$ defined as~${
\lp=n+1+\ell}$, $ 1\leq\lp\leq n-1$, $ \kp=n+1+k$, $ 2\leq\kp\leq n$.
Analogously, we replace the summation over strictly negative
indices $-n\leq i\leq \ell$ and $k\leq j\leq -1$ in~\r{gl1} by the summation
over strictly positive indices $\ip$ and $\jp$ such that
$
\ip=i+n+1$, $ 1\leq \ip\leq \lp$, $ \jp=j+n+1$, $ \kp\leq\jp\leq n $.
Moreover, instead of the sets of Bethe parameters $\bt^s$ we consider
the shifted sets of Bethe parameters
\begin{equation}\label{gl10}
\btt^s=\bt^{ n-s}-c s=
\bigl\{t^{ n-s}_1-c s,\dots,t^{ n-s}_{r_{n-s}}-c s\bigr\},\qquad s=1,\dots,n-1 .
\end{equation}
These sets will be partitioned in the sum of the partitions
$\{\btt^s_{\so},\btt^s_{\st},\btt^s_{\sth}\}\vdash \btt^s$ with cardinalities
according to \r{gl3} equal to
\[
|\btt^s_{\so}|=\begin{cases}
\Theta(s-\ip), &s<\lp,\\
0, &s\geq \lp,
\end{cases}\qquad
|\btt^s_{\sth}|=\begin{cases}
0, &s<\kp,\\
\Theta(\jp-s-1), &s\geq \kp.
\end{cases}
\]

For $2\leq\kp\leq n$, we introduce the functions $\hlam_\kp(u)$
and $\halp_{\kp}(u)$ by the equations
\begin{gather*}
\hlam_\kp(u)=\frac{1}{\lambda_{n+1-\kp}(u+c(\kp-1))}
\prod_{s=1}^{\kp-1}\frac{\lambda_{n+1-s}(u+cs)}{\lambda_{n+1-s}(u+c(s-1))},\\
\halp_{\kp}(u)=\frac{\hlam_\kp(u)}{\hlam_{\kp+1}(u)}=
\frac{\lambda_{n-\kp}(u+c\kp)}{\lambda_{n+1-\kp}(u+c\kp)}
 ={\alpha_{n-\kp}(u+c\kp)} .
\end{gather*}
Then according to the relation \r{lam} which connects the values of the functions
$\lambda_k(u)$ for negative and positive values of the index $k$
and definition \r{gl10} one gets
\begin{equation}\label{gl13}
\lambda_k(z)=\hlam_\kp(z_n)\qquad \mbox{and}\qquad
\halp_s(\btt^s)=\alpha_{n-s}\bigl(\bt^{n-s}\bigr) .
\end{equation}
Using all these definitions, one can calculate
\begin{gather*}
\frac{1}{\psi_\ell\bigl(z;\bt_{\st}\bigr)}
\prod_{s=-\ell}^{-i-1}\Omr\bigl(\bt^s_{\so},\bt^s_{{\st}}|\bt^{s-1}_{\st},\bt^{s+1}_{\st}\bigr)
=\frac{(-1)^{|\btt^{\ip-1}|-|\btt^{\ip}_{\st}|}}
{g\bigl(z_n,\btt^{\lp-1}_{\st}\bigr) h\bigl(\btt^{\lp},z_n\bigr)}
\prod_{s=\ip}^{\lp-1}\Omr\bigl(\btt^s_{\so},\btt^s_{{\st}}|\btt^{s-1}_{\st},\btt^{s+1}_{\st}\bigr) ,
\\
\frac{1}{\phi_k\bigl(z;\bt_{\st}\bigr)}
\prod_{s=-j}^{-k-1}\alpha_s\bigl(\bt^s_{\sth}\bigr)
\Oml\bigl(\bt^s_{\so},\bt^s_{{\st}}|\bt^{s-1}_{\st},\bt^{s+1}_{\st}\bigr)
\\
\qquad= \frac{(-1)^{|\btt^{\jp-1}_{\st}|-|\btt^{\jp}|}}
{h\bigl(z_n,\btt^{\kp-1}\bigr) g\bigl(\btt^{\kp}_{\st},z_n\bigr)}
\prod_{s=\kp}^{\jp-1}\halp_s\bigl(\btt^s_{\sth}\bigr)
\Oml\bigl(\btt^s_{\st},\btt^s_{{\sth}}|\btt^{s-1}_{\st},\btt^{s+1}_{\st}\bigr),
\end{gather*}
and the recurrence relations \r{gl1}
for the indices $\ell$ and $k$ such that
$-n\leq \ell<k<0$ becomes
\begin{gather}
\BB'\bigl(\varnothing,\bigl\{\bt^s\bigr\}_1^{-k-1},\bigl\{\bt^s,z_s\bigr\}_{-k}^{-\ell-1},\bigl\{\bt^s\bigr\}_{-\ell}^{n-1}\bigr)\nonumber\\
\qquad= \sum_{\ip=1}^\lp\sum_{\jp=\kp}^n \sum_{\rm part}
\frac{T_{i,j}(z)
\cdot\BB'\bigl(\varnothing,\bigl\{\bt^s\bigr\}_{1}^{-j-1},
\bigl\{\bt^s_{\st}\bigr\}_{-j}^{-k-1},
\bigl\{\bt^s\bigr\}_{-k}^{-\ell-1},\bigl\{\bt^s_{\st}\bigr\}_{-\ell}^{-i-1},\bigl\{\bt^s\bigr\}_{-i}^{n-1}\bigr)}
{\hlam_\kp(z_n) g\bigl(z_n,\btt^{\lp-1}_{\st}\bigr) h\bigl(\btt^\lp,z_n\bigr)
 h\bigl(z_n,\btt^{\kp-1}\bigr) g\bigl(\btt^\kp_{\st},z_n\bigr)}\nonumber\\
\phantom{\qquad=}{}\times
\prod_{s=\ip}^{\lp-1}\Omr\bigl(\btt^s_{\so},\btt^s_{\st}|\btt^{s-1}_{\st},\btt^{s+1}_{\st}\bigr)
\prod_{s=\kp}^{\jp-1}\halp_s\bigl(\btt^s_{\sth}\bigr)\
\Oml\bigl(\btt^s_{\st},\btt^s_{\sth}|\btt^{s-1}_{\st},\btt^{s+1}_{\st}\bigr),\label{gl-recpp}
\end{gather}
where
\smash{$
 \BB'\bigl(\bt\bigr) = (-1)^{\sum_{s=1}^{n-2} |\bt^{s}| |\bt^{s+1}| + \sum_{s=1}^{n-1} |\bt^{s}|} \BB\bigl(\bt\bigr)$}.

Let us define the monodromy matrix entries
$\hat{\TT}_{\ip,\jp}(z)$ for all $\ip,\jp=1,\dots,n$ defined by embedding~\r{em-sh} together with a shift of the spectral parameter as
\begin{equation}\label{TT mon}
\hat{\TT}_{\ip,\jp}(z_n)= T_{i,j}(z) \qquad\text{with}\quad \begin{cases} i=\ip-n-1 , \\ j=\jp-n-1 . \end{cases}
\end{equation}
The corresponding eigenvalues of $\hat{\TT}_{k'k'}(z_n)$ on the vacuum vector $\rvac$ coincide with $\hat{\lambda}_{k'}(z_n)$ from \r{gl13}.

Note that the coefficients in the recursion relation \r{gl-recpp} have exactly the same form as \r{gl-red1}, if one replaces $z_n\to z$, $\btt^s\to\bt^s$, $\hlam_{k'}(z_n)\to
\lambda_{k'}(z)$, $\hat{\TT}_{i',j'}(z_n)\to \TT_{i',j'}(z)$.

Thus, the Bethe vector $\BB\bigl(\bt\bigr)$ relates to Bethe vectors $\hBB\bigl(\bt\bigr)$ constructed from the monodromy matrix~\smash{$\hat{\TT}(z)$} \r{TT mon} by relation
\[
\hBB\bigl(\varnothing,\bigl\{\bt^{ n-s}-cs-c\kappa_n\bigr\}_1^{n-1}\bigr)=
(-1)^{\sum_{s=1}^{n-2} |\bt^{s}| |\bt^{s+1}| + \sum_{s=1}^{n-1} |\bt^{s}|} \BB\bigl(\varnothing,\bigl\{\bt^{s}\bigr\}_1^{n-1}\bigr) .
\]

Such relation was studied in the paper \cite{LPRS-New}, where it was shown that the vector Bethe $\hBB\bigl(\bt\bigr)$ is \mbox{constructed} from an inverted and transposed $\mathfrak{gl}_n$-type monodromy
matrix, i.e., \smash{$\hat{\TT}(z_n) = \bigl(\TT(z)^{-1}\bigr)^{\rm t}$}.

\subsection[Embedding Y(o\_2a+1)otimes
 Y(gl\_n-a)-> Y(o\_2n+1)]{Embedding $ \boldsymbol{Y(\mathfrak{o}_{2a+1})\otimes
 Y(\mathfrak{gl}_{n-a})\hookrightarrow Y(\mathfrak{o}_{2n+1})}$
}\label{sect:red-o.x.gl}

In this subsection, we study the recurrence relations \r{main-rr} in the case
when the set $\bt^a=\varnothing$ for some $a>0$.
\begin{prop}\label{emb-pr}
The off-shell Bethe vectors
\smash{$\BBo\bigl(\bigl\{\bt^s\bigr\}_0^{a-1},\varnothing,\bigl\{\bt^s\bigr\}_{a+1}^{n-1}\bigr)$} factorizes
into the product of $\mathfrak{o}_{2a+1}$-type pre-Bethe vector
\smash{$\cBo\bigl(\bigl\{\bt^s\bigr\}_0^{a-1}\bigr)$} and $\mathfrak{gl}_{n-a}$-type
pre-Bethe vectors $\cBg\bigl(\bigl\{\bt^s\bigr\}_{a+1}^{n-1}\bigr)$
\begin{equation}\label{em1}
\BBo\bigl(\bigl\{\bt^s\bigr\}_0^{a-1},\varnothing,\bigl\{\bt^s\bigr\}_{a+1}^{n-1}\bigr)=\cBo\bigl(\bigl\{\bt^s\bigr\}_0^{a-1}\bigr)
\cdot \cBg\bigl(\bigl\{\bt^s\bigr\}_{a+1}^{n-1}\bigr) |0\rangle .
\end{equation}
\end{prop}

\begin{proof}
The proof of this proposition is analogous to the proof of Proposition~\ref{emb-gl}.
We consider the recurrence
relations \r{main-rr} for three different ranges of the indices $\ell$ and $k$:
$-a\leq \ell<k\leq a$, $a+1\leq \ell<k\leq n$ and $-n\leq \ell<k\leq -a-1$.
These ranges for $k$ and $\ell$ ensure that the set $\bt^a$ remains empty.
Since
$\cBo(\varnothing) = \cBg(\varnothing) = 1$, the recurrence relations will prove the statement of the proposition by induction on the cardinalities.
\begin{itemize}\itemsep=0pt
\item $-a\leq \ell<k\leq a$. According to \r{part-so}, the cardinality of the empty subset $\bt^a_{\so}$
is given by the step function
$\Theta(-i-a-1)$, which must vanish. It immediately results that the summation
in the recurrence relation \eqref{main-rr} is restricted to the range $-a\leq i\leq \ell$.
Analogously, the cardinality of the empty subset $\bt^a_{\sth}$ given by \r{part-sth}
is equal to $\Theta(j-a-1)$ which restricts the range in the summation
over $j$ to $k\leq j\leq a$. This means that the recurrence relation
for \smash{$\Z^k_\ell\cdot
\BB\bigl(\bt\bigr)|_{\bt^a=\varnothing}$}
implies only the monodromy entries $T_{i,j}(z)$ for $-a\leq i\leq \ell$ and
$k\leq j\leq a$. The recurrence relation $\Z^a_{-a}\cdot
\BB\bigl(\bt\bigr)|_{\bt^a=\varnothing}$ will have only one term corresponding to the action of
monodromy entry $T_{-a,a}(z)$ on the off-shell Bethe vector.

Moreover, the action of the zero mode operators $\sT_{\ell+1,\ell}$ for
$\ell=0,1,\dots,a-1$ onto the Bethe vector~\r{em1}
\begin{gather*}
\sT_{\ell+1,\ell}\cdot\BBo\bigl(\bigl\{\bt^s\bigr\}_0^{a-1},\varnothing,\bigl\{\bt^s\bigr\}_{a+1}^{n-1}\bigr)\nonumber\\
\qquad=
\sum_{\rm part}\bigl(\vk_{\ell+1} \alpha_\ell\bigl(\bt^\ell_{\so}\bigr)
\Oml_\ell\bigl(\bt^\ell_{\st},\bt^\ell_{\so}|\bt^{\ell-1},\bt^{\ell+1}\bigr)
-\vk_\ell \Omr_\ell\bigl(\bt^\ell_{\so},\bt^\ell_{\st}|\bt^{\ell-1},\bt^{\ell+1}\bigr)\bigr)\nonumber\\
\phantom{\qquad=}{}\times
\BBo\bigl(\bigl\{\bt^s\bigr\}_0^{\ell-1},\bt^\ell_{\st},\bigl\{\bt^s\bigr\}_{\ell+1}^{a-1},
\varnothing,\bigl\{\bt^s\bigr\}_{a+1}^{n-1}\bigr)\label{zmac-red}
\end{gather*}
does not affect the sets of Bethe parameters \smash{$\bigl\{\bt^s\bigr\}_{a+1}^{n-1}$}
and so coincides with the action of the zero mode operators
onto the Bethe vectors in $\mathfrak{o}_{2a+1}$-invariant model.
This shows that the recurrence relations for $\Z^k_\ell\cdot
\BB\bigl(\bt\bigr)|_{\bt^a=\varnothing}$ for $-a\leq \ell< k\leq a$
 coincides with the $\mathfrak{o}_{2a+1}$ recurrence relation for the vector
\smash{$\BBo\bigl(\bigl\{\bt\bigr\}_0^{a-1}\bigr)$}.
\item $a+1\leq \ell<k\leq n$. In this case, the cardinalities of the empty subsets
$\bt^a_{\so}=\bt^a_{\sth}=\varnothing$ given by the equalities
\r{part-so} and \r{part-sth} yield a restriction only for the index $i\geq a+1$.
The summation over $i$ and $j$
in the right-hand side of the recurrence relation for \smash{$\Z^k_\ell\cdot
\BB\bigl(\bt\bigr)|_{\bt^a=\varnothing}$} in this case will be restricted
to the range $a+1\leq i\leq \ell<k\leq j\leq n$.
The Bethe vector $\Z^n_{a+1}\cdot \BB\bigl(\bt\bigr)$ will be proportional
to the action of the monodromy matrix entry \smash{$T_{a+1,n}(z)\cdot\BB\bigl(\bt\bigr)$} and the
action of the zero mode operators~$\sT_{\ell+1,\ell}$ for
$\ell=a+1,\dots,n-1$ onto Bethe vectors \r{em1} coincide~with $\mathfrak{gl}_{n-a}$-type actions~\r{Az6}.
The~recurrence relation~\r{main-rr} in this case
will not affect the sets of~Bethe parameters \smash{$\bigl\{\bt^s\bigr\}_0^{a-1}$} and will
coincides~with the~$\mathfrak{gl}_{n-a}$-type
recurrence relation for the Bethe vector~\smash{$\Z^k_\ell\cdot\BBg\bigl(\bigl\{\bt^s\bigr\}_{a+1}^{n-1}\bigr)$},
see~\eqref{gl-red1}.
\item $-n\leq \ell<k\leq -a-1$. Using again \r{part-so} and \r{part-sth}, we can
observe that the summation over $i$ and $j$ in the right-hand side of the recurrence relation
for the Bethe vector $\Z^k_\ell\cdot
\BB\bigl(\bt\bigr)|_{\bt^a=\varnothing}$ will be restricted to the range
$-n\leq i\leq \ell<k\leq j\leq -a-1$. Using arguments of Section~\ref{sect:red-gln},
one can show that the corresponding
recurrence relation for $-n\leq \ell< k\leq -a-1$ will be equivalent to
the $\mathfrak{gl}_{n-a}$-type recurrence relation for
$\hBB\bigl(\bt\bigr)$, where $\hat\BB\bigl(\bt\bigr)$ is the Bethe vector symmetric
to $\BB\bigl(\bt\bigr)$
in the sense of the symmetry introduced in the paper \cite{LPRS-New}.\hfill \qed
\end{itemize}\renewcommand{\qed}{}
\end{proof}
\begin{rem}
When $a=n-1$, the recurrence relations described in the second and third
items are absent and the recurrence relations described in the first item
yield the reduction over rank $n\to n-1$ of the recurrence relations for the
off-shell Bethe vectors in $\mathfrak{o}_{2n-1}$-invariant integrable models.
\end{rem}

\begin{rem}[splitting property]
Using the statement of Proposition~\ref{emb-pr} and
the rectangular recurrence relations \r{main-rr} for the
off-shell Bethe vectors in $\mathfrak{o}_{2n+1}$-invariant
integrable models, one can prove a~splitting
property for these vectors, similar to~\r{emg3}. Again, the splitting property
is a direct consequence of the presentation of the Bethe vectors within
the projection method~\cite{LP2}.
\end{rem}

\begin{exam}\label{Exam41}
Proposition~{\rm\ref{emb-pr}} for the case when $a=0$ allows
to get Bethe vectors in $\mathfrak{o}_{2n+1}$-invariant integrable model
from the Bethe vectors of the $\mathfrak{gl}_{n}$-invariant model.
Below we illustrate this in the case of $\mathfrak{o}_5$-invariant
model.
According to the general recurrence relation \r{ex3} for the Bethe vector
$\Z_0^1\cdot\BB\bigl(\bt\bigr)$ in the case when $n=2$ and $\bt^0=\varnothing$,
we obtain for the Bethe vector $\BBo\bigl(z,\bt^1\bigr)$ the presentation
\begin{equation*}
\BBo\bigl(z,\bt^1\bigr)=\frac{1}{\lambda_1(z) g\bigl(\bt^1,z\bigr)}
\left(
T_{0,1}(z)\cdot\BBg\bigl(\varnothing,\bt^1\bigr)+
\sum_{a=1}^{r_1}g\bigl(t^1_a,z\bigr) \alpha_1\bigl(t^1_a\bigr)
\frac{g\bigl(\bt^1_a,t^1_a\bigr)}{h\bigl(t^1_a,\bt^1_a\bigr)} T_{0,2}(z)\cdot\BBg\bigl(\varnothing,\bt^1_a\bigr)
\right),
\end{equation*}
where set $\bt^1_a=\bt^1\setminus\bigl\{t^1_a\bigr\}$ and
$\BB\bigl(\varnothing,\bt^1\bigr)$ are Bethe vectors in
$\mathfrak{gl}_2$-integrable model.
 Here we restrict the summation over the index $i$ to the case when cardinality of the set $\bt^0$
is equal to $0$. When cardinality $\big|\bt^1\big|=1$ and $\bt^1=\bigl\{t^1\bigr\}$ this formula
yields a presentation of the Bethe vector $\BB\bigl(t^0,t^1\bigr)$
\begin{equation}\label{o5ex}
\BBo\bigl(t^0,t^1\bigr)=\frac{1}{\lambda_1\bigl(t^0\bigr) \lambda_2\bigl(t^1\bigr) g\bigl(t^1,t^0\bigr)}
\bigl(T_{0,1}(t^0) T_{1,2}\bigl(t^1\bigr)+g\bigl(t^1,t^0\bigr) T_{0,2}\bigl(t^0\bigr) T_{1,1}\bigl(t^1\bigr)\bigr)\rvec ,
\end{equation}
 where we rename $z\to t^0$. This formula has the same form as for the analogous Bethe vector $\BBg\bigl(t^1,t^2\bigr)$.

We can apply the recurrence relation \r{ex3} to the Bethe vector \r{o5ex}
once again to obtain the presentation for the Bethe vector $\BBo\bigl(\bigl\{t^0_1,t^0_2\bigr\},t^1\bigr)$
in the form
\begin{align*}
\BBo\bigl(\bigl\{t^0_1,t^0_2\bigr\},t^1\bigr)={}&
\bigl({\lambda_1\bigl(t^0_1\bigr) \lambda_1\bigl(t^0_2\bigr) \lambda_2\bigl(t^1\bigr)
\fgo\bigl(t^0_1+c/2,t^0_2\bigr) g\bigl(t^1,t^0_1\bigr) g\bigl(t^1,t^0_2\bigr)}\bigr)^{-1}\Bigl(T_{0,1}\bigl(t^0_1\bigr) T_{0,1}\bigl(t^0_2\bigr) T_{1,2}\bigl(t^1\bigr)\\
&+
g\bigl(t^1,t^0_2\bigr)\Bigl(T_{0,1}\bigl(t^0_1\bigr) T_{0,2}\bigl(t^0_2\bigr) T_{1,1}\bigl(t^1\bigr)+ g\bigl(t^0_2,t^0_1+c/2\bigr)\Bigl(T_{-2,1}\bigl(t^0_1\bigr) T_{1,1}\bigl(t^0_2\bigr) T_{2,2}\bigl(t^1\bigr)\\
&+
h\bigl(t^1,t^0_2\bigr) T_{-1,1}\bigl(t^0_1\bigr) T_{1,2}\bigl(t^1\bigr) T_{1,1}\bigl(t^0_2\bigr)
\Bigr)+g\bigl(t^1,t^0_1\bigr) h\bigl(t^1,t^0_2\bigr)\Bigl(T_{0,2}\bigl(t^0_1\bigr) \\
& \phantom{+}{}\times T_{0,1}\bigl(t^0_2\bigr) T_{1,1}\bigl(t^1\bigr)+
g\bigl(t^0_2,t^0_1+c/2\bigr)
T_{-1,2}\bigl(t^0_1\bigr) T_{1,1}\bigl(t^0_2\bigr) T_{1,1}\bigl(t^1\bigr)\Bigr)\Bigr)\Bigr)\rvec .
\end{align*}

This Bethe vector already differs greatly from the corresponding example for the
Bethe vector in the case of $\mathfrak{gl}_3$. Indeed, due to the
recurrence relations given in the Corollary~{\rm\ref{cor1}}, the latter can be written
in the following form
\begin{align*}
\BBg\bigl(\bigl\{t^1_1,t^1_2\bigr\},t^2\bigr)={}&\bigl({\lambda_2\bigl(t^1_1\bigr) \lambda_2\bigl(t^1_2\bigr)
\lambda_3\bigl(t^2\bigr)}
{h\bigl(t^1_1,t^1_2\bigr) h\bigl(t^1_2,t^1_1\bigr)} {g\bigl(t^2,t^1_1\bigr) g\bigl(t^2,t^1_2\bigr)}\bigr)^{-1}\\
&\times\bigl(T_{1,2}\bigl(t^1_2\bigr) T_{1,2}\bigl(t^1_1\bigr) T_{2,3}\bigl(t^2\bigr) +\lambda_2\bigl(t^2\bigr) g\bigl(t^2,t^1_1\bigr) f\bigl(t^1_1,t^1_2\bigr) T_{1,3}\bigl(t^1_1\bigr)
T_{1,2}\bigl(t^1_2\bigr)\\
&\phantom{\times}{}+\lambda_2\bigl(t^2\bigr) g\bigl(t^2,t^1_2\bigr) f\bigl(t^1_2,t^1_1\bigr) T_{1,3}\bigl(t^1_2\bigr) T_{1,2}\bigl(t^1_1\bigr)
\bigr) |0\rangle .
\end{align*}
\end{exam}

\section{Conclusion}\label{conclu}

Using the zero modes method, we describe in this paper the rectangular recurrence
relations for the off-shell Bethe vectors in $\mathfrak{gl}_n$- and
$\mathfrak{o}_{2n+1}$-invariant integrable models. These relations
are presented as sums of the actions of the monodromy entries
$T_{i,j}(z)$ with $i<j$ on the off-shell Bethe vectors $\BB\bigl(\bt\bigr)$ and sums
over partitions of the sets of Bethe parameters. When one of the
set $\bt^s$ becomes empty, the recurrence relations given by
Theorem~\ref{main-th} describe a factorization of the Bethe vectors
and relates to the symmetry of $\mathfrak{gl}_n$-type Bethe vectors found in the
paper \cite{LPRS-New}.

Now that the recurrence relations are obtained, one can deduce recurrence relations for scalar products, and for their building block, the so-called highest coefficients. One can also compute the norm of the on-shell Bethe vectors and express it as a determinant. Calculations are under review, and the results will appear soon.

Finally, the outcome obtained in this paper can be generalized for the Bethe vectors
in $\mathfrak{sp}_{2n}$- and $\mathfrak{so}_{2n}$-invariant integrable models. The procedure is similar to the one presented here and is
currently under investigation.
The issue will be published elsewhere.

\appendix

\section[Proof of the Theorem 3.1]{Proof of Theorem~\ref{rec-pr}}\label{ApA}

The proof of this theorem is based on the following lemma.
\begin{lem}\label{l-pr}
We recall that $\bar u^s=\bigl\{\bar t^s,z\bigr\}$.
The off-shell Bethe vector
\smash{$\BB\bigl(\bigl\{\bt^s\bigr\}_1^{\ell-1},\bigl\{\bu^s\bigr\}_\ell^{n-1}\bigr)$} for $1\leq \ell< n$
can be presented in the form
\begin{align}
\Z_{\ell}^n\cdot \BB\bigl(\bt\bigr)={}&
\frac{1}{\mu_\ell^n\bigl(z;\bt\bigr)}
\sum_{i=1}^\ell \sum_{\rm part} g\bigl(z,\bt^{\ell-1}_{\so}\bigr)
\prod_{p=i}^{\ell-1}\Omr\bigl(\bt^p_{\so},\bt^p_{\st}|\bt^{p-1}_{\st},\bt^{p+1}_{\st}\bigr)\nonumber
\\
&\times T_{i,n}(z)\cdot\BB\bigl(\bigl\{\bt^s\bigr\}_1^{i-1},\bigl\{\bt^s_{\st}\bigr\}_i^{\ell-1},
\bigl\{\bt^s\bigr\}_\ell^{n-1}\bigr) ,\label{Az8}
\end{align}
where the sum goes over partitions $\bigl\{\bt^s_{\so},\bt^s_{\st}\bigr\}\vdash\bt^s$
with cardinalities $\big|\bt^s_{\so}\big|=1$ for all $s=i,\dots,\ell-1$, $\bt^\ell$ is not
 partitioned with $\bt^\ell_{\st}=\bt^\ell$, $\bt^{i-1}_{\st} = \bt^{i-1}$, $\bt^{i-1}_{\so} = \varnothing$, and $\bt^0_{\so} = \varnothing$.
\end{lem}

\begin{proof}
 We prove relation \r{Az8} through an induction over $\ell$.
Equality
\r{Az4} can be considered as the base of this induction since for $\ell=1$
there are no partitions in the right-hand side of \r{Az8} and it becomes identical to
\r{Az4}.

Assuming that equality \r{Az8} is valid for
some fixed $\ell< n-1$, we multiply it by the
normalization factor $\mu^n_\ell\bigl(z;\bt\bigr)$ and apply to both sides the
zero mode operator $\sT_{\ell+1,\ell}$.
Using the action of the zero modes~\r{Az6}, we get from the left-hand side of \r{Az8}
\begin{gather}
-\vk_{\ell} \mu^n_{\ell+1}\bigl(z;\bt\bigr) \Z_{\ell+1}^n\cdot \BB\bigl(\bt\bigr) +\sum_{\rm part}\bigl(\vk_{\ell+1} \alpha_{\ell}\bigl(\bt^\ell_{\so}\bigr)
\Oml\bigl(\bt^\ell_{\st},\bt^\ell_{\so}|\bt^{\ell-1},\bt^{\ell+1}\bigr)-
\vk_\ell f\bigl(\bt^\ell_{\so},z\bigr)
\Omr\bigl(\bt^\ell_{\so},\bt^\ell_{\st}|\bt^{\ell-1},\bt^{\ell+1}\bigr)\bigr)\nonumber\\
\qquad \times \mu^n_\ell\bigl(z;\bt^{\ell-1},\bt^\ell_{\st},\bt^{n-1}\bigr)
\Z_{\ell}^n\cdot \BB\bigl(\bigl\{\bt^s\bigr\}_1^{\ell-1},\bt^\ell_{\st},\bigl\{\bt^s\bigr\}_{\ell+1}^{n-1}\bigr) .\label{Az15}
\end{gather}
Here the sum goes over partitions $\bigl\{\bt^\ell_{\so},\bt^\ell_{\st}\bigr\}\vdash
\bt^\ell$ with cardinality $\big|\bt^\ell_{\so}\big|=1$. In the last line of \eqref{Az15}, we wrote explicitly the arguments
of the function $\mu^n_\ell\bigl(z;\bt^{\ell-1},\bt^\ell_{\st},\bt^{n-1}\bigr)$ for clarification.
Using the induction assumption, relation \r{Az15} rewrites
\begin{gather}
-\vk_{\ell} \mu^n_{\ell+1}\bigl(z;\bt\bigr) \Z_{\ell+1}^n\cdot \BB\bigl(\bt\bigr) +\sum_{\rm part}
\bigl(\vk_{\ell+1} \alpha_{\ell}\bigl(\bt^\ell_{\so}\bigr)
\Oml\bigl(\bt^\ell_{\st},\bt^\ell_{\so}|\bt^{\ell-1},\bt^{\ell+1}\bigr)-
\vk_\ell f\bigl(\bt^\ell_{\so},z\bigr)
\Omr\bigl(\bt^\ell_{\so},\bt^\ell_{\st}|\bt^{\ell-1},\bt^{\ell+1}\bigr)\bigr)\nonumber\\
\qquad \times
\sum_{i=1}^\ell g\bigl(z,\bt^{\ell-1}_{\so}\bigr)
\prod_{p=i}^{\ell-1}\Omr\bigl(\bt^p_{\so},\bt^p_{\st}|\bt^{p-1}_{\st},\bt^{p+1}_{\st}\bigr)
 T_{i,n}(z)\cdot\BB\bigl(\bigl\{\bt^s\bigr\}_1^{i-1},\bigl\{\bt^s_{\st}\bigr\}_i^{\ell},
\bigl\{\bt^s\bigr\}_{\ell+1}^{n-1}\bigr) .\label{Az15a}
\end{gather}
Here the sum goes over partitions $\bigl\{\bt^s_{\so},\bt^s_{\st}\bigr\}\vdash\bt^s$
with cardinalities $\big|\bt^s_{\so}\big|=1$ for all $s=i,\dots,\ell$, $\bt^{i-1}_{\st} = \bt^{i-1}$, and $\bt^{i-1}_{\so} = \varnothing$.

To get \r{Az15} and \r{Az15a}, we have used the equalities
\begin{gather*}
\Oml\bigl(\bt^\ell,z|\bt^{\ell-1},\bu^{\ell+1}\bigr)=0 ,\qquad
\Omr\bigl(z,\bt^\ell |\bt^{\ell-1},\bu^{\ell+1}\bigr)=\frac{g\bigl(z,\bt^\ell\bigr)}{h\bigl(\bt^\ell,z\bigr)}
\frac{h\bigl(\bt^{\ell+1},z\bigr)}{g\bigl(z,\bt^{\ell-1}\bigr)} ,\\
\Oml\bigl(\bigl\{\bt^\ell_{\st},z\bigr\},\bt^\ell_{\so}|\bt^{\ell-1},\bu^{\ell+1}\bigr)=
\frac{1}{h\bigl(\bt^\ell_{\so},z\bigr)} \Oml\bigl(\bt^\ell_{\st},\bt^\ell_{\so}|\bt^{\ell-1},\bt^{\ell+1}\bigr) ,\\
\Omr\bigl(\bt^\ell_{\so},\bigl\{z,\bt^\ell_{\st}\bigr\}|\bt^{\ell-1},\bu^{\ell+1}\bigr)
=g\bigl(\bt^\ell_{\so},z\bigr) \Omr\bigl(\bt^\ell_{\so},\bt^\ell_{\st}|\bt^{\ell-1},\bt^{\ell+1}\bigr) ,\\
\mu^n_\ell\bigl(z;\bt\bigr) \frac{g\bigl(z,\bt^\ell\bigr)}{h\bigl(\bt^\ell,z\bigr)}
\frac{h\bigl(\bt^{\ell+1},z\bigr)}{g\bigl(z,\bt^{\ell-1}\bigr)} = \mu^n_{\ell+1}\bigl(z;\bt\bigr) ,\qquad
\mu^n_\ell\bigl(z;\bt\bigr)=h\bigl(\bt^\ell_{\so},z\bigr) \mu^n_\ell\bigl(z;\bt^{\ell-1},\bt^\ell_{\st},\bt^{n-1}\bigr) .
\end{gather*}

Using the commutation relations
$
[\sT_{\ell+1,\ell},T_{i,n}(z)]=-\vk_{\ell}\ \delta_{\ell,i}\ T_{\ell+1,n}(z)
$,
the action of the zero mode $\sT_{\ell+1,\ell}$ on the right-hand side of the induction assumption \r{Az8}
reads
\begin{gather}
-\vk_{\ell} T_{\ell+1,n}(z)\cdot\BB\bigl(\bt\bigr)+
\sum_{i=1}^{\ell}\sum_{\rm part}g\bigl(z,\bt^{\ell-1}_{\so}\bigr)\prod_{p=i}^{\ell-2}
\Omr\bigl(\bt^p_{\so},\bt^p_{\st}|\bt^{p-1}_{\st},\bt^{p+1}_{\st}\bigr)
\Omr\bigl(\bt^{\ell-1}_{\so},\bt^{\ell-1}_{\st}|\bt^{\ell-2}_{\st},\bigl\{\bt^\ell_{\so},\bt^\ell_{\st}\bigr\}\bigr)\nonumber\\
\qquad\times
\bigl(\vk_{\ell+1} \alpha_{\ell}\bigl(\bt^\ell_{\so}\bigr)
\Oml\bigl(\bt^\ell_{\st},\bt^\ell_{\so}|\bt^{\ell-1}_{\st},\bt^{\ell+1}\bigr)-
\vk_\ell
\Omr\bigl(\bt^\ell_{\so},\bt^\ell_{\st}|\bt^{\ell-1}_{\st},\bt^{\ell+1}\bigr)\bigr)\nonumber\\
\qquad\times T_{i,n}(z)
\cdot \BB\bigl(\bigl\{\bt^s\bigr\}_1^{i-1},\bigl\{\bt^s_{\st}\bigr\}_i^{\ell},\bigl\{\bt^s\bigr\}_{\ell+1}^{n-1}\bigr)\bigr) .\label{Az16}
\end{gather}
Here the sum over partitions is the same as in \r{Az15a}.

Let us compare the coefficients of the twisting parameters $\vk_{\ell+1}$
and $\vk_\ell$ in \r{Az15a} and \r{Az16}.
Due to the relations
\begin{gather*}
\Oml\bigl(\bt^\ell_{\st},\bt^\ell_{\so}|\bigl\{\bt^{\ell-1}_{\so},\bt^{\ell-1}_{\st}\bigr\},\bt^{\ell+1}\bigr)=
h\bigl(\bt^\ell_{\so},\bt^{\ell-1}_{\so}\bigr)
\Oml\bigl(\bt^\ell_{\st},\bt^\ell_{\so}|\bt^{\ell-1}_{\st},\bt^{\ell+1}\bigr),\\
 \Omr\bigl(\bt^{\ell-1}_{\so},\bt^{\ell-1}_{\st}|\bt^{\ell-2}_{\st},\bigl\{\bt^{\ell}_{\so},\bt^\ell_{\st}\bigr\}\bigr)=
h(\bt^\ell_{\so},\bt^{\ell-1}_{\so})
\Omr\bigl(\bt^{\ell-1}_{\so},\bt^{\ell-1}_{\st}|\bt^{\ell-2}_{\st},\bt^{\ell}_{\st}\bigr),
\end{gather*}
the terms proportional to $\vk_{\ell+1}$
in \r{Az15a} and in \r{Az16} are equal, and cancel each other when we equate~\r{Az15a} and \r{Az16}. Then $\chi_\ell$ factorizes globally, and we get a relation with no explicit dependence of the $\chi_i$ parameters:
this is a manifestation of the
principle described in Remark~\ref{rem: chi homogeneity}.

On the other hand, using in \r{Az15a} the equality
\begin{equation*}
\Omr\bigl(\bt^\ell_{\so},\bt^\ell_{\st}|\bigl\{\bt^{\ell-1}_{\so},\bt^{\ell-1}_{\st}\bigr\},\bt^{\ell+1}\bigr)=
\frac{1}{g\bigl(\bt^\ell_{\so},\bt^{\ell-1}_{\so}\bigr)}
\Omr\bigl(\bt^\ell_{\so},\bt^\ell_{\st}|\bt^{\ell-1}_{\st},\bt^{\ell+1}\bigr)
\end{equation*}
we conclude that the equality between \r{Az15a} and \r{Az16} is
equivalent to the relation \r{Az8} at $\ell\to \ell+1$ due to the identities
\begin{gather}
g\bigl(z,\bt^{\ell-1}_{\so}\bigr)\left(h\bigl(\bt^\ell_{\so},\bt^{\ell-1}_{\so}\bigr)-
\frac{f\bigl(\bt^\ell_{\so},z\bigr)}{g\bigl(\bt^\ell_{\so},\bt^{\ell-1}_{\so}\bigr)}\right)\nonumber\\
\qquad=
\frac{g\bigl(z,\bt^{\ell-1}_{\so}\bigr)}
{g\bigl(\bt^\ell_{\so},\bt^{\ell-1}_{\so}\bigr)}\bigl(f\bigl(\bt^\ell_{\so},\bt^{\ell-1}_{\so}\bigr)-
f\bigl(\bt^\ell_{\so},z\bigr)\bigr)\nonumber
\\
\qquad=\frac{g\bigl(z,\bt^{\ell-1}_{\so}\bigr)}
{g\bigl(\bt^\ell_{\so},\bt^{\ell-1}_{\so}\bigr)}\bigl(g\bigl(\bt^\ell_{\so},\bt^{\ell-1}_{\so}\bigr)-
g\bigl(\bt^\ell_{\so},z\bigr)\bigr)=
\frac{g\bigl(z,\bt^{\ell-1}_{\so}\bigr) g\bigl(\bt^\ell_{\so},\bt^{\ell-1}_{\so}\bigr)
g\bigl(\bt^\ell_{\so},z\bigr)}
{g\bigl(\bt^\ell_{\so},\bt^{\ell-1}_{\so}\bigr) g\bigl(\bt^{\ell-1}_{\so},z\bigr)}=
g\bigl(z,\bt^{\ell}_{\so}\bigr)\label{simp}
\end{gather}
 for $i<\ell$ and
$ 1 - f(\bt^{i}_{\so}, z) = g(z, \bt^{i}_{\so})$
for $i=\ell$.
This ends the inductive proof of relation~\eqref{Az8}. \end{proof}

{\bf End of theorem's proof.}
To finish the proof of Theorem~\ref{rec-pr}
one has to perform an inductive proof over $k$ for the recurrence
relation \r{Az19}. We will consider the induction step $k+1\to k$ taking
as induction base the just proved equality \r{Az8}.

Let us assume that equality \r{Az19} is valid for $\Z_{\ell}^{k+1}$ for some $k<n$.
The induction proof means
 that this assumption should lead to the equality \r{Az19} for $\Z_{\ell}^{k}$.
To perform the induction step, we multiply both sides of \r{Az19} at $k+1$ by the function~$\mu^{k+1}_\ell\bigl(z;\bt\bigr)$ and act on this relation by the zero mode operator~$\sT_{k+1,k}$.
Using the equalities
\begin{gather*}
\Oml\bigl(\bt^k,z|\bu^{k-1},\bt^{k+1}\bigr)=\frac{g\bigl(\bt^k,z\bigr)}{h\bigl(z,\bt^k\bigr)}
\frac{h\bigl(z,\bt^{k-1}\bigr)}{g\bigl(\bt^{k+1},z\bigr)} ,\qquad
\Omr\bigl(z,\bt^k|\bu^{k-1},\bt^{k+1}\bigr)=0 ,\\
\Oml\bigl(\bigl\{\bt^k_{\st},z\bigr\},\bt^k_{\sth}|\bu^{k-1},\bt^{k+1}\bigr)=
g\bigl(z,\bt^k_{\sth}\bigr) \Oml\bigl(\bt^k_{\st},\bt^k_{\sth}|\bt^{k-1},\bt^{k+1}\bigr) ,\\
\Omr\bigl(\bt^k_{\sth},\bigl\{\bt^k_{\st},z\bigr\}|\bu^{k-1},\bt^{k+1}\bigr)=
\frac{1}{h\bigl(z,\bt^k_{\sth}\bigr)} \Omr\bigl(\bt^k_{\sth},\bt^k_{\st}|\bt^{k-1},\bt^{k+1}\bigr) ,\\
\alpha_k(z) \mu^{k+1}_\ell\bigl(z;\bt\bigr) \frac{g\bigl(\bt^k,z\bigr)}{h\bigl(z,\bt^k\bigr)}
\frac{h\bigl(z,\bt^{k-1}\bigr)}{g\bigl(\bt^{k+1},z\bigr)}=\mu^k_\ell\bigl(z;\bt\bigr) ,\qquad
\mu^{k+1}_\ell\bigl(z;\bt\bigr)=h\bigl(z,\bt^k_{\sth}\bigr) \mu^{k+1}_\ell\bigl(z;\bt^{l-1},\bt^\ell,\bt^k_{\st},\bt^{k+1}\bigr),
\end{gather*}
one can write the left-hand side of the resulting relation as follows:
\begin{gather}
\mu^{k+1}_\ell\bigl(z;\bt\bigr) \sT_{k+1,k}\cdot
\BB\bigl(\bigl\{\bt^s\bigr\}_{1}^{\ell-1},\bigl\{\bu^s\bigr\}_{\ell}^{k},\bigl\{\bt^s\bigr\}_{k+1}^{n-1}\bigr)\nonumber\\
\qquad
=\vk_{k+1} \mu^k_\ell\bigl(z;\bt\bigr) \Z^k_\ell\cdot\BB\bigl(\bt\bigr)+ \sum_{i=1}^\ell\sum_{j=k+1}^n\sum_{\rm part}
\bigl(\vk_{k+1} \alpha_k\bigl(\bt^k_{\sth}\bigr) \Oml\bigl(\bt^k_{\st},\bt^k_{\sth}|\bt^{k-1},\bt^{k+1}\bigr)
f\bigl(z,\bt^k_{\sth}\bigr)\nonumber\\
\phantom{\qquad
=}{}-\vk_k \Omr\bigl(\bt^k_{\sth},\bt^k_{\st}|\bt^{k-1},\bt^{k+1}\bigr)\bigr) \Xi^{\ell,k+1}_{i,j}\bigl(z;\bigl\{\bt^s\bigr\}_{1}^{k-1},
\bt^k_{\st},\bigl\{\bt^s\bigr\}_{k+1}^{n-1}\bigr)\nonumber\\
\phantom{\qquad
=-}{}\times
T_{i,j}(z)\cdot\BB\bigl(\bigl\{\bt^s\bigr\}_{1}^{i-1},\bigl\{\bt^s_{\st}\bigr\}_{i}^{\ell-1},\bigl\{\bt^s\bigr\}_{\ell}^{k-1},
\bigl\{\bt^s_{\st}\bigr\}_{k}^{j-1},\bigl\{\bt^s\bigr\}_{j}^{n-1}\bigr) .\label{Az20}
\end{gather}
Here for further convenience, the partition resulting from the action of the zero mode operator
$\sT_{k+1,k}$ is noted $\bigl\{\bt^k_{\st},\bt^k_{\sth}\bigr\}\vdash\bt^k$
with cardinality $\big|\bt^k_{\sth}\big|=1$.

On the other hand, the right-hand side of the same equality takes the form
\begin{gather}
\vk_{k+1} \sum_{i=1}^\ell\sum_{\rm part}\Xi^{\ell,k+1}_{i,k+1}(z,\bt) T_{i,k}(z)\cdot \BB\bigl(\bt_{\st}\bigr)+
\sum_{i=1}^\ell\sum_{j=k+1}^n\sum_{\rm part} \Xi^{\ell,k+1}_{i,j}\bigl(z;\bt\bigr)
\big|_{\{\bt^k_{\st},\bt^k_{\sth}\}\vdash\bt^k}\nonumber\\
\qquad\times
\bigl(\vk_{k+1} \alpha_k(\bt^k_{\sth})
\Oml(\bt^k_{\st},\bt^k_{\sth}|\bt^{k-1},\bt^{k+1}_{\st})
-\vk_k \Omr(\bt^k_{\sth},\bt^k_{\st}|\bt^{k-1},\bt^{k+1}_{\st})\bigr)\nonumber\\
\qquad\times
T_{i,j}(z)\cdot\BB\bigl(\bigl\{\bt^s\bigr\}_{1}^{i-1},\bigl\{\bt^s_{\st}\bigr\}_{i}^{\ell-1},\bigl\{\bt^s\bigr\}_{\ell}^{k-1},
\bigl\{\bt^s_{\st}\bigr\}_{k}^{j-1},\bigl\{\bt^s\bigr\}_{j}^{n-1}\bigr) ,\label{Az21}
\end{gather}
where we used the commutation relation \r{Az5} and took into account that
the summation index $i$ in \r{Az21} satisfies the inequalities $1\leq i\leq \ell< k$.
Due to the formulas
\begin{gather*}
\Xi^{\ell,k+1}_{i,k+1}\bigl(z;\bt\bigr)=\Xi^{\ell,k}_{i,k}\bigl(z;\bt\bigr)=g\bigl(z,\bt^{\ell-1}_{\so}\bigr)
\prod_{p=i}^{\ell-1}\Omr\bigl(\bt^{p}_{\so},\bt^{p}_{\st}|\bt^{p-1}_{\st},\bt^{p+1}_{\st}\bigr) ,\\
\Xi^{\ell,k+1}_{i,j}\bigl(z;\bt\bigr)
\big|_{\{\bt^k_{\st},\bt^k_{\sth}\}\vdash\bt^k}=
h\bigl(\bt^{k+1}_{\sth},\bt^k_{\sth}\bigr)
\Xi^{\ell,k+1}_{i,j}\bigl(z;\bigl\{\bt^s\bigr\}_{1}^{k-1},
\bt^k_{\st},\bigl\{\bt^s\bigr\}_{k+1}^{n-1}\bigr),
\\
\Omr\bigl(\bt^{k}_{\sth},\bt^{k}_{\st}\big|\bt^{k-1},\bigl\{\bt^{k+1}_{\st},\bt^{k+1}_{\sth}\bigr\}\bigr) =
h\bigl(\bt^{k+1}_{\sth},\bt^k_{\sth}\bigr)
\Omr\bigl(\bt^{k}_{\sth},\bt^{k}_{\st}\big|\bt^{k-1},\bt^{k+1}_{\st}\bigr) ,
\\
\Oml\bigl(\bt^{k}_{\st},\bt^{k}_{\sth}\big|\bt^{k-1},\bigl\{\bt^{k+1}_{\st},\bt^{k+1}_{\sth}\bigr\}\bigr) =
\frac{1}{g\bigl(\bt^{k+1}_{\sth},\bt^k_{\sth}\bigr)}
\Oml\bigl(\bt^{k}_{\st},\bt^{k}_{\sth}\big|\bt^{k-1},\bt^{k+1}_{\st}\bigr)
\end{gather*}
and the identity for the rational functions
\begin{equation*}
g\bigl(\bt^{k+1}_{\sth},z\bigr)\left(h\bigl(\bt^{k+1}_{\sth},\bt^k_{\sth}\bigr)-
\frac{f\bigl(z,\bt^k_{\sth}\bigr)}{g\bigl(\bt^{k+1}_{\sth},\bt^k_{\sth}\bigr)}\right)=g\bigl(\bt^k_{\sth},z\bigr),
\end{equation*}
in the equality between \r{Az20} and \r{Az21} coefficients at $\vk_k$ cancel each other and coefficients at $\vk_{k+1}$ yield~\r{Az19}.
This finishes the inductive proof of Theorem~\ref{rec-pr}.\qed

{\bf Alternative proof.}
Theorem~\ref{rec-pr} can also be proven using an induction $\ell\to \ell+1$ and
taking as induction base the relation \eqref{Az18}, coming from following alternative lemma.
\begin{lem}\label{r-pr}
The off-shell Bethe vector $\Z_{1}^{k}\cdot \BB\bigl(\bt\bigr)=\BB\bigl(\bigl\{\bu^s\bigr\}_1^{k-1},\bigl\{\bt^{s}\bigr\}_k^{n-1}\bigr)$ for $1<k\leq n$
can be presented in the form
\begin{align}
\Z_{1}^{k}\cdot \BB\bigl(\bt\bigr)={}&
\frac1{\mu_1^k\bigl(z;\bt\bigr)}
\sum_{j=k}^n \sum_{\rm part} g\bigl(\bt^{k}_{\sth},z\bigr)
\prod_{p=k}^{j-1}\alpha_p\bigl(\bt^p_{\sth}\bigr)
\Oml\bigl(\bt^p_{\st},\bt^p_{\sth}|\bt^{p-1}_{\st},\bt^{p+1}_{\st}\bigr)\nonumber\\
&\times T_{1,j}(z)\cdot\BB\bigl(\bigl\{\bt^s\bigr\}_1^{k-1},\bigl\{\bt^s_{\st}\bigr\}_{k}^{j-1},
\bigl\{\bt^s\bigr\}_j^{n-1}\bigr),\label{Az18}
\end{align}
where the sum goes over partitions $ \bigl\{\bt^s_{\st},\bt^s_{\sth}\bigr\}\vdash\bt^s$
with cardinalities $\big|\bt^s_{\sth}\big|=1$ for all $s=k,\dots,j-1$, the set $\bt^{k-1}$ is not partitioned, $\bt^{k-1}_{\st}=\bt^{k-1}$, and $\bt^{n}_{\sth} = \varnothing$.
\end{lem}
The proofs of Lemma \ref{r-pr} and the end of recursion for Theorem~\ref{rec-pr} are similar as above, and we do not reproduce them in the present paper.

\section[Sketch of the proof of Theorem 4.1]{Sketch of the proof of Theorem~\ref{main-th}}\label{ApB}

The proof of this theorem follows the method
described in Appendix~\ref{ApA}. The starting point is the simple
recurrence relation \r{Bz6} for the Bethe vector \smash{$\BB\bigl(\bigl\{\bw^s\bigr\}_0^{n-1}\bigr)$}.

The recurrence relations \r{main-rr} become the relations \r{B1-1}, \r{B1-2}, and
\r{B1-3} for $n=1$. They were proved in \cite{LPRS19}. To prove Theorem~\ref{main-th}, it is sufficient to consider the cases when $n>1$.

Applying the zero mode operator $\sT_{n,n-1}$ to \r{Bz6}
and using the commutation relations between
the zero modes and the monodromy entries \r{Bz0} as well as the action of the
zero mode operators on the off-shell Bethe vectors~\r{Bz8}, we get
a relation which involves terms proportional to $\chi_n$ and terms proportional to $\chi_{n-1}$. Since the relation involves only Bethe vectors, monodromy matrix entries $T_{ij}(z)$ and eigenvalues $\lambda_i(z)$, due
 to Remark~\ref{rem: chi homogeneity}, the coefficients of these two independent twisting parameters yield
 two recurrence relations, one for the Bethe vectors
\smash{$\BB\bigl(\bigl\{\bw^s\bigr\}_0^{n-2},\bu^{n-1}\bigr)=\Z^n_{-n+1}\cdot\BB\bigl(\bt\bigr)$} and one for the Bethe vector
\smash{$\BB\bigl(\bigl\{\bw^s\bigr\}_0^{n-2},\bv^{n-1}\bigr)=\Z^{n-1}_{-n}\cdot\BB\bigl(\bt\bigr)$}.
These relations are
\begin{gather}
\Z_{-n+1}^n\cdot\BB\bigl(\bt\bigr)=\frac{1}
{\mu_{-n+1}^n\bigl(z;\bt\bigr)} \sum_{i=-n}^{-n+1}\sum_{\rm part}
g\bigl(\bt^{n-1}_{\so},z_{n-1}\bigr) \Omr\bigl(\bt^{n-1}_{\so},\bt^{n-1}_{\st}|\bt^{n-2},\varnothing\bigr)\nonumber
\\
\phantom{\Z_{-n+1}^n\cdot\BB\bigl(\bt\bigr)=}{}\times
T_{i,n}(z)\cdot\BB\bigl(\bigl\{\bt^s\bigr\}_0^{n-2},\bt^{n-1}_{\st}\bigr) ,\label{Bz16}
\\
\Z_{-n}^{n-1}\cdot\BB\bigl(\bt\bigr)=\frac{1}
{\mu_{-n}^{n-1}\bigl(z;\bt\bigr)} \sum_{j=n-1}^n\sum_{\rm part}
g\bigl(\bt^{n-1}_{\sth},z\bigr) \alpha_{n-1}\bigl(\bt^{n-1}_{\sth}\bigr)
\Oml\bigl(\bt^{n-1}_{\st},\bt^{n-1}_{\sth}|\bt^{n-2},\varnothing\bigr)\nonumber\\
\phantom{\Z_{-n+1}^n\cdot\BB\bigl(\bt\bigr)=}{}\times
T_{-n,j}(z)\cdot\BB\bigl(\bigl\{\bt^s\bigr\}_0^{n-2},\bt^{n-1}_{\st}\bigr) ,\label{Bz17}
\end{gather}
where $\big|\bt^{n-1}_{\so}\big|=\Theta(-i-n)$ and $\big|\bt^{n-1}_{\sth}\big|=\Theta(j-n)$.

These recurrence relations
allows to get an inductive proof of the formula \r{main-rr} similarly to the one presented in Appendix \ref{ApA}.
The only differences for the different steps of this inductive proof are in the different ranges for the indices $\ell$ and $k$ that one has to consider,
and for which different identities are needed.

Some of these identities will be different from those used in the proof of
the recurrence relations for~$\mathfrak{gl}_n$ Bethe vectors because of the following
mechanism. To consider the inductive step $k+1\to k$ or~${\ell\to \ell+1}$
in the recurrence relation \r{main-rr} we will apply the zero mode operator
$\sT_{k+1,k}$ or $\sT_{\ell+1,\ell}$ to the inductive assumption
recurrence relations which correspond to the indices $k+1$
and $\ell$ respectively. Since the action of the zero mode operator
$\sT_{k+1,k}$ also parts the set of Bethe parameters
$\bigl\{\bt^k_{\sso},\bt^k_{\sst}\bigr\}\vdash\bt^k$, the partition of $\bt^k$ first by the induction assumption and then by the zero mode action, or vice-versa,
 may lead to different splittings in the left and right hand sides of the resulting recurrence
relation. In the left-hand side the set $\bt^k$ first parts into subsets
$\bigl\{\bt^k_{\sso},\bt^k_{\sst}\bigr\} $ through the action of the zero mode, and then the subset $\bt^k_{\sst}$ is partitioned
into subsets $\bt^k_{\sst}\vdash\bigl\{\bt^k_{\so},\bt^k_{\st}\bigr\}$ according to
the induction assumption. On the other hand, in the right-hand side the set
$\bt^k$ first parts into subsets
\smash{$\bigl\{\bt^k_{\so},\bt^k_{\st}\bigr\} $} through the induction assumption, and then the subset $\bt^k_{\st}$ is partitioned
into subsets $\bigl\{\bt^k_{\sso},\bt^k_{\sst}\bigr\}\vdash\bt^k_{\st}$
by the action of the zero mode.
If the value of the index $j$ is such that according to \r{part-sth}
the subset $\bt^k_{\so}$ is not empty, then the resulting equality after
action of $\sT_{k+1,k}$ should be symmetrized over the subsets
$\bt^k_{\sso}$ and $\bt^k_{\so}$ both having the cardinality 1.
It will make appear the subset $\bt^k_{\sso}\cup\bt^k_{\so}$, which may have cardinality 2, hence
 the cardinality 2 subset
$\bt^k_{\so}$ in the sum over partitions in the recurrence relation \r{main-rr}.
This phenomena does not happen for $\mathfrak{gl}_n$ Bethe vectors
but is present for the inductive proof of the recurrence relations
for the Bethe vectors of other algebra series.

Referring to the calculations presented in Appendix~\ref{ApA} for more details,
we describe below the different ranges of the indices $\ell$, $k$ and the
corresponding identities which should be used at each of the inductive step.
Let us divide the whole domain $-n\leq \ell< k\leq n$
of the values of the indices $\ell$ and $k$
 into three subdomains:
$1\leq-\ell,k\leq n$, $\ 0\leq \ell< k\leq n$, and
$-n\leq \ell< k\leq 0$.

For the first subdomain $1\leq-\ell, k\leq n$, the calculations
can be performed as a sequence of the
following steps
\begin{itemize}\itemsep=0pt
\item $\Z^n_{-n}\cdot\BB\bigl(\bt\bigr)\to \Z^n_{-n+1}\cdot\BB\bigl(\bt\bigr)$
and $\Z^n_{-n}\cdot\BB\bigl(\bt\bigr)\to \Z^{n-1}_{-n}\cdot\BB\bigl(\bt\bigr)$.
The recurrence relations for the Bethe~vectors
\smash{$\BB\bigl(\bigl\{\bw^s\bigr\}_0^{n-2},\bu^{n-1}\bigr)=\Z^n_{-n+1}\cdot\BB\bigl(\bt\bigr)$}
and
\smash{$\BB\bigl(\bigl\{\bw^s\bigr\}_0^{n-2},\bv^{n-1}\bigr)=\Z^{n-1}_{-n}\cdot\BB\bigl(\bt\bigr)$}
are given by the formulas \r{Bz16} and \r{Bz17} above.

\item $\Z^n_{-n+1}\cdot\BB\bigl(\bt\bigr)\to \Z^n_\ell\cdot\BB\bigl(\bt\bigr)$, $1\leq -\ell\leq n-1$.
The recurrence relations \r{main-rr} is proved in these cases
by induction over $\ell$ starting from the recurrence relation \r{Bz16}.
In this case, we will not need to perform the symmetrization described above
and the only identity which will be necessary to consider these cases is the
simple identity \r{simp}.

\item $\Z^n_{\ell}\cdot\BB\bigl(\bt\bigr)\to \Z^k_\ell\cdot\BB\bigl(\bt\bigr)$, $-\ell\leq k\leq n-1$.
These cases are proved by induction over $k$ starting from the
recurrence relation for the Bethe vector $\Z^n_{\ell}\cdot\BB\bigl(\bt\bigr)$
which is already proved at the previous step. This proof will
require to use the simple identity \r{simp} and a more complicated
identity
\begin{equation}\label{ApB1}
\mathop{\operatorname{ Sym}}_{y_1,y_2}
\sk{\frac{h(q,y_1)}{g(y_1,x)} \frac{g(y_1,y_2)}{h(y_2,y_1)}-
\frac{h(y_1,x)}{g(q,y_1)} \frac{g(y_2,y_1)}{h(y_1,y_2)}}=0,
\end{equation}
where for any expression $E(y_1,y_2)$ we define
\[
\mathop{\operatorname{ Sym}}_{y_1,y_2} E(y_1,y_2)
=E(y_1,y_2)+E(y_2,y_1) .
\]
\item $\Z^{n-1}_{-n}\cdot\BB\bigl(\bt\bigr)\to \Z^k_{-n}\cdot\BB\bigl(\bt\bigr)$, $1\leq k\leq n-1$.
The induction proof of the recurrence
relations will require only the simple identity \r{simp}.

\item $\Z^k_{-n}\cdot\BB\bigl(\bt\bigr)\to \Z^k_{\ell}\cdot\BB\bigl(\bt\bigr)$, $k\leq -\ell \leq n-1$.
 For these cases, the starting point will be the recurrence relation for
 the Bethe vector $\Z^k_{-n}\cdot\BB\bigl(\bt\bigr)$ proved at the
 previous step. We will need
 for these cases the identities \r{simp} and \r{ApB1}.
\end{itemize}

For the second domain $0\leq l< k\leq n$, the
 most simple way to prove the recurrence relations~\r{main-rr}
can be depicted as the sequence
of the following steps.

\begin{itemize}\itemsep=0pt
\item The first step of the induction $\Z^n_{-1}\cdot\BB\bigl(\bt\bigr)\to \Z^n_0\cdot\BB\bigl(\bt\bigr)$
is simple. It will not require any complicated rational functions identities,
but fixes the function $\psi_0\bigl(z;\bt\bigr)$ in \r{psi}.
 \item The next step $\Z^n_{0}\cdot\BB\bigl(\bt\bigr)\to \Z^n_1\cdot\BB\bigl(\bt\bigr)$ is also particular.
It is the first step for which partitions with~${|\bar t^0_{\so}|=2}$ appear.
 To get the recurrence relation, we use the identities
\begin{equation}\label{ApB0}
 f(t,z)-\frac{f(z,t)}{f(z_1,t)}=-g(z_0,t)
 \end{equation}
 and
\begin{equation}\label{ApB4}
 \mathop{\operatorname{ Sym}}_{y_1,y_2}\left(\frac{g(y_2,z_0)}{g(x,y_1)}
 \left(
 \fgo(y_1,y_2) f(x,y_1) \frac{f(z,y_1)}{f(z_1,y_1)} - \fgo(y_2,y_1)
 \right)\right)= g(z,\bar y) h(x,z) ,
 \end{equation}
 where $\bar y=\{y_1,y_2\}$.\item Then, to prove $\Z^n_\ell\cdot\BB\bigl(\bt\bigr)$, $\ell>1$, starting from $\Z^n_1\cdot\BB\bigl(\bt\bigr)$,
 we use the identities
\begin{equation}\label{ApB2}
g(z,\bar x)\sk{h(t,\bar x)-\frac{f(t,z)}
{g(t,\bar x)}-h(t,z)}=g(z,t)
\end{equation}
and
\begin{equation}\label{ApB3}
\mathop{\operatorname{ Sym}}_{y_1,y_2}\sk{h(y_2,z)
\sk{\frac{h(y_1,\bar x)}{g(q,y_1)} \frac{g(y_2,y_1)}{h(y_1,y_2)}
- f(y_1,z)
\frac{h(q,y_1)}{g(y_1,\bar x)} \frac{g(y_1,y_2)}{h(y_2,y_1)}}}=
\frac{g(z,\bar y)h(q,z)}{g(z,\bar x)} ,
\end{equation}
where $\bar x=\{x_1,x_2\}$ and $\bar y=\{y_1,y_2\}$.
\item For the steps $\Z^n_{\ell}\cdot\BB\bigl(\bt\bigr)\to \Z^k_\ell\cdot\BB\bigl(\bt\bigr)$, $0\leq \ell<k\leq n$,
besides the simple identity \r{simp}, the identity~\r{ApB1} should be used.
\end{itemize}

Finally, for the third domain $-n\leq l< k\leq 0$, the recurrence relations \r{main-rr}
 is proved in several steps, involving the two particular cases corresponding to
$k=0$ and $k=-1$.
Among the recurrence relations corresponding to this domain there are the
 so called shifted recurrence relations, when
the sets of Bethe parameters $\bt^s$ are extended by the shifted parameter
$z_s=z-c(s-1/2)$.

\begin{itemize}\itemsep=0pt
\item $\Z^1_{-n}\cdot\BB\bigl(\bt\bigr)\to \Z^0_{-n}\cdot\BB\bigl(\bt\bigr)$.
This step is simple. No complicated identities should be used, but it fixes the
function $\phi_0\bigl(z;\bt\bigr)$ in \r{phi}.
\item $\Z^0_{-n}\cdot\BB\bigl(\bt\bigr)\to \Z^{-1}_{-n}\cdot\BB\bigl(\bt\bigr)$.
Here the proof relies on the identities
\r{ApB0} and an identity equivalent to \r{ApB4}.
 \item For $\Z^{-1}_{-n}\cdot\BB\bigl(\bt\bigr)\to \Z^k_{-n}\cdot\BB\bigl(\bt\bigr)$, $-n<k< -1$,
 we need identities which appear to be equivalent to~\r{ApB2} and \r{ApB3}.
\item $\Z^k_{-n}\cdot\BB\bigl(\bt\bigr)\to \Z^k_{\ell}\cdot\BB\bigl(\bt\bigr)$, $-n\leq \ell< k\leq 0$.
This final step will require identities \r{ApB1}.
\end{itemize}

\subsection*{Acknowledgements}
We are grateful to Alexander Molev for fruitful discussions on embeddings in Yangian algebras.
We would like to acknowledge the anonymous referees for their numerous relevant remarks, which contributed to improving the paper.

S.P.~acknowledges the support of the PAUSE Programme and hospitality at LAPTh where
this work was done.
The research of A.L.\ was supported by Beijing Natural Science Foundation (IS24006) and Beijing Talent Program.
A.L.\ is also grateful to the CNRS PHYSIQUE for support during his visit to
Annecy in the course of this investigation.



\pdfbookmark[1]{References}{ref}
\LastPageEnding

\end{document}